\documentclass[aop]{imsart}

\RequirePackage{amsthm,amsmath,amsfonts,amssymb}
\RequirePackage[numbers,sort&compress]{natbib}
\RequirePackage[colorlinks,citecolor=blue,urlcolor=blue,hypertexnames=false]{hyperref}
\RequirePackage{graphicx}

\usepackage{mathrsfs}
\usepackage{url}
\usepackage{enumitem}
\usepackage[usenames, dvipsnames]{color}
\usepackage{verbatim}
\usepackage{float}
\usepackage{tikz}
\usepackage{tikz-cd}
\usepackage{subcaption}
\usepackage{pgfplots}
\usepackage{adjustbox}
\usepackage{bbm}
\usepackage{xcolor}
\usepackage{pdfsync}
\usepackage{cleveref}
\usepackage{extpfeil}
\usepackage{ytableau}
\usepackage{thmtools}

\startlocaldefs

\theoremstyle{plain}
\newtheorem{prop}{Proposition}[section]
\newtheorem{proposition}[prop]{Proposition}
\newtheorem{theorem}[prop]{Theorem}
\newtheorem{corollary}[prop]{Corollary}

\newtheorem{claim}[prop]{Claim}
\newtheorem{lemma}[prop]{Lemma}
\newtheorem{fact}[prop]{Fact}
\newtheorem{example}[prop]{Example}

\theoremstyle{remark}
\newtheorem{defi}{Definition}[section]
\newtheorem{definition}[defi]{Definition}
\newtheorem{rmk}{Remark}[section]

\numberwithin{equation}{section}

\renewcommand\Im{{\operatorname{Im}}}

\newcommand\rank{{\operatorname{rank}}}

\newcommand\R{{\mathbb{R}}}
\newcommand\N{{\mathbb{N}}}

\renewcommand\P{{\mathbf{P}}}
\newcommand\E{{\mathbf{E}}}

\newcommand\diag{{\operatorname{diag}}}

\newcommand\Z{{\mathbb{Z}}}

\newcommand\F{{\mathbb{F}}}

\newcommand\al{\alpha}
\newcommand\la{\lambda}

\newcommand\Bv{{\mathbf v}}

\renewcommand\Pr{{\mathbf P }}

\newcommand\CC{{\mathcal C}}
\newcommand\CE{{\mathcal E}}

\newcommand\CM{{\mathcal M}}

\newcommand\BBZ {{\mathbb Z}}

\newcommand\tl{{\tilde \lambda}}

\newcommand{\tX}{\tilde{X}}
\newcommand{\tbeta}{\tilde{\beta}}

\renewcommand \small {\scriptsize}

\newcommand\eps{\varepsilon}

\newcommand\lang{\langle}
\newcommand\rang{\rangle}

\newcommand\cok{\mathbf{Cok}}
\newcommand\Aut{\operatorname{Aut}}

\newcommand{\ra}{\rightarrow}

\newcommand\Hom{\operatorname{Hom}}
\newcommand\Sur{\operatorname{Sur}}

\newcommand\GL{\operatorname{GL}}

\newcommand{\hH}{\hat{H}}

\newcommand\bq{\begin{equation}}
\newcommand\eq{\end{equation}}

\newcommand\Inj{\operatorname{Inj}}
\newcommand\Sig{\operatorname{Sig}}

\newcommand{\tth}{^{th}}

\newcommand{\Mat}{\operatorname{Mat}}

\newcommand{\cL}{\mathcal{L}}
\newcommand{\tcL}{\tilde{\cL}}
\newcommand{\hcL}{\hat{\cL}}
\newcommand{\Y}{\mathbb{Y}}

\DeclareMathOperator{\bSig}{\overline{Sig}}
\newcommand{\bZ}{\overline{\mathbb{Z}}}

\DeclareMathOperator{\corank}{corank}

\DeclareSymbolFont{bbold}{U}{bbold}{m}{n}
\DeclareSymbolFontAlphabet{\mathbbold}{bbold}
\newcommand{\bbone}{\mathbbold{1}}

\newcommand{\meas}{M}
\newcommand{\tmeas}{\tilde{\meas}}

\DeclarePairedDelimiter{\abs}{\lvert}{\rvert} 
\DeclarePairedDelimiter{\floor}{\lfloor}{\rfloor}

\DeclarePairedDelimiter{\near}{\lfloor}{\rceil}

\endlocaldefs

\begin{document}

\begin{frontmatter}
 \title{Rank fluctuations of matrix products and a moment method for growing groups}
\runtitle{Rank fluctuations and a moment method for growing groups}

\begin{aug}
\author[A]{\fnms{Hoi H.}~\snm{Nguyen} \ead[label=e1]{nguyen.1261@osu.edu}},
\author[B]{\fnms{Roger}~\snm{Van Peski}\ead[label=e2]{rogervanpeski@gmail.com}}

\address[A]{Department of Mathematics,
Ohio State University \printead[presep={ ,\ }]{e1}}

\address[B]{Department of Mathematics,
Columbia University\printead[presep={,\ }]{e2}}
\end{aug}

\begin{abstract}
We consider the $p$-primary part of the cokernel $G_n = \cok(A_{k} \cdots A_2 A_1)[p^\infty]$ of a product of independent $n \times n$ random integer matrices with iid entries from generic nondegenerate distributions, in the regime where both $n$ and $k$ are sent to $\infty$ simultaneously. In this regime we show that the $G_n$ converges universally to the reflecting Poisson sea, an interacting particle system constructed in \cite{van2023reflecting}, at the level of $1$-point marginals. In particular, $\operatorname{corank}(A_{k} \cdots A_2 A_1 \pmod{p}) \sim \log_p k$, and its fluctuations are $O(1)$ and converge to a discrete random variable defined in \cite{van2023local}.

The main difference with previous works studying cokernels of random matrices is that $G_n$ does not converge to a random finite group; for instance, the $p$-rank of $G_n$ diverges. This means that the usual moment method for random groups does not apply. Instead, we proceed by proving a `rescaled moment method' theorem applicable to a general sequence of random groups of growing size. This result establishes that fluctuations of $p$-ranks and other statistics still converge to limit random variables, provided that certain rescaled moments $\E[\#\Hom(G_n,H)]/C(n,H)$ converge.
\end{abstract}



\end{frontmatter}

 


 \maketitle

\tableofcontents

\section{Introduction}

\subsection{Preface} This paper concerns three kinds of objects: random abelian groups, discrete random matrices, and continuous-time interacting particle systems. Our goal is to (1) develop a general-purpose `rescaled moment method' for computing limiting fluctuations of random groups of growing size, and (2) apply this machinery to establish universality of a connection between limits of random matrix cokernels and a new interacting particle system constructed in \cite{van2023reflecting}.

While this interacting particle system limit is new, relations between random abelian groups and discrete random matrices are simple and well-known: given a nonsingular matrix $A \in \Mat_n(\Z)$, its cokernel
\begin{equation*}
\cok(A) := \Z^n / A \Z^n
\end{equation*}
is a finite abelian group, so a random matrix yields a random group. By the structure theorem, such a group is isomorphic to $\bigoplus_{i=1}^n \Z/a_i \Z$ for some collection of positive integers $a_i$. Often such random groups converge to a universal random (finite, abelian) group as $n \to \infty$, regardless of the finer details of the matrix entry distribution. Such universal distributions arise in many contexts (either proved or unproven): 
\begin{itemize}
\item The Cohen-Lenstra heuristics \cite{cohen-lenstra} in arithmetic statistics, see Friedman-Washington \cite{friedman-washington}, Ellenberg-Venkatesh-Westerland \cite{ellenberg2011modeling,ellenberg2016homological}, Bhargava-Kane-Lenstra-Poonen-Rains \cite{bhargava2013modeling}, Maples \cite{maples}, Wood \cite{wood2018cohen,W1,wood2023probability}, Lipnowski-Sawin-Tsimmerman \cite{lipnowski2020cohen};
\vskip .1in
\item Jacobians or sandpile groups of random graphs, see Clancy-Kaplan-Leake-Payne-Wood \cite{clancy2015cohen}, Wood \cite{W0}, M\'esz\'aros \cite{meszaros2020distribution}, Nguyen-Wood \cite{nguyen2022local}; 
\vskip .1in 
\item (Co)homology groups of random simplicial complexes, see Kahle \cite{kahle2014topology}, Kahle-Lutz-Newman-Parsons \cite{kahle2020cohen}, M\'esz\'aros \cite{meszaros2023cohen,meszaros2024bounds,meszaros20242}.
\end{itemize}
There is relatively new---but by now standard---machinery to prove universal convergence of matrix cokernels to such distributions, the \emph{moment method} first developed for finite abelian groups by Wood \cite{W0}. 

However, there remain many interesting sequences of random groups $(G_n)_{n \to \infty}$ which do \emph{not} converge to a random finite group $G$, but for which one might still hope to say something about the asymptotics of quantities associated to $G_n$ such as its $p$-rank\footnote{Recall that this means the rank of $G_n/pG_n$ as an $\F_p$-vector space.}. Our motivating example is cokernels of products of random matrices; we mention a few others at the end of the Introduction.

In the related and simpler $p$-adic setting, recent works \cite{van2023local,van2023reflecting} considered the sequence of cokernels 
\begin{equation}\label{eq:cok_product_first}
\cok(A_k A_{k-1} \cdots A_1), k \in \Z_{\geq 0}
\end{equation}
of products of iid matrices $A_1,A_2,\ldots$. These works viewed \eqref{eq:cok_product_first} as defining a discrete-time stochastic process on the set of finite $p$-abelian groups $\bigoplus_{i=1}^n \Z/p^{a_i} \Z$, or equivalently on the numbers $\{a_i\}_{1 \leq i \leq n}$, with $k$ playing the role of time. This process cannot converge to a limiting random group as $k$ increases, since by multiplicativity of the determinant 
\begin{equation*}
|\cok(A_k A_{k-1} \cdots A_1)| = \prod_{i=1}^k |\cok(A_i)|.
\end{equation*}
Unexpectedly, the result of \cite{van2023reflecting} finds that the numbers $a_i$ evolve in a particularly simple manner when $i$ is large, according to an interacting particle system dubbed the \emph{reflecting Poisson sea} there. In particular as $n,k \to \infty$, they converged to the distribution of the reflecting Poisson sea at a fixed time (\cite[Theorem 10.1]{van2023local} and \cite[Theorem 8.2]{van2023reflecting}). Note this is not a limit of random finite groups, which cannot exist, but rather a certain limit of the numbers $a_i$ as a point process. This limit as both $n,k \to \infty$, given in \cite[Theorem 10.1 and 10.2]{van2023local}, has only been shown in special examples, using delicate techniques originating from harmonic analysis on $p$-adic groups and integrable probability which are specific to these cases\footnote{Oddly, universality could be shown for the dynamics by which the $a_i$ evolve in time in \cite[Theorem 1.4]{van2023reflecting}, but not for their distribution at a fixed time.}. 

Meanwhile, if the number of products $k$ is fixed independent of the matrix size $n$, then previous work by the authors \cite{nguyen2022universality} shows that for any nondegenerate distribution on the matrix entries, $\cok(A_k \cdots A_1)$ converges to a limiting universal random group as the matrix size $n$ is sent to $\infty$, as was established for a single matrix by Wood \cite{W1}. This work showed convergence to a random group using the moment method, but as soon as $k$ is sent to $\infty$, the limiting cokernel no longer exists as a random finite group and the standard moment method no longer applies. 

In our first main result (\Cref{thm:matrix_product_intro}) below, we nonetheless manage to prove universality of the limiting results of \cite[Theorem 10.1]{van2023local} in the setting of generic iid matrix entries, despite the fact that the limiting cokernels do not exist. To do so we prove a second main result (\Cref{thm:group_convergence_intro}), a `rescaled moment method' applicable to sequences of random abelian groups $(G_n)_{n \geq 1}$ which do not converge to a limit random group $G$. More specifically, for finite abelian $p$-groups $G_n$, our result determines the limits and joint distribution of fluctuations of the $p$-ranks $\rank(p^{i-1}G_n), i =1,2,\ldots$. These are natural statistics to study: not only does the sequence $\rank(p^{i-1}G_n), i =1,2,\ldots$ determine an abelian $p$-group $G_n$ up to isomorphism, but the truncated sequences $\rank(p^{i-1}G_n), i=1,\ldots,d$ determine the isomorphism type of the quotients $G_n/p^dG_n$, for any $d$. \Cref{thm:group_convergence_intro} is a general result for random abelian groups, and broadly applicable beyond the matrix product case.

\subsection{Universal discrete limits for matrix product cokernels} Fix a prime $p$, for a finite abelian group $G$ let $G[p^\infty]$ denote its $p$-Sylow subgroup, also called its $p^\infty$-torsion part. Consider the random abelian $p$-group
\begin{equation}
G_n = \cok(A_{k}^{(n)} \cdots A_{2}^{(n)} A_{1}^{(n)})[p^\infty]
\end{equation}
where $A_i^{(n)}$ are iid random matrices in $\Mat_n(\Z)$, and $k=k(n)$ is some sequence depending on $n$ which goes to $\infty$ along with $n$. 

Fix a positive integer $d$. The sequence $(\rank(p^{i-1}G_n))_{1 \leq i \leq d}$ mentioned before is a decreasing sequence of $d$ nonnegative integers, i.e. an integer partition of length at most $d$. We denote the set of such partitions by $\Y_d$, and refer to the integers themselves as the \emph{parts} of the partition. In fact, defining 
\begin{equation}
G_{\la} := \bigoplus_{i} \Z/p^{\la_i}\Z
\end{equation}
and the conjugate partition as in \Cref{fig:ferrers}, we see that the set of all isomorphism classes of $p^d$-torsion abelian groups is just $\{G_{\la'}: \la \in \Y_d\}$.

\begin{figure}[H]
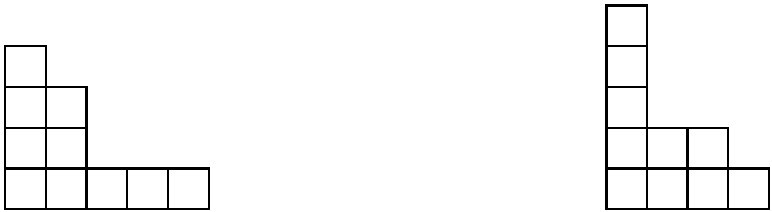

\begin{equation*}
    \raisebox{-.5\height}{\begin{ytableau}
        *(white) \\ 
        *(white) & *(white) \\ 
        *(white) & *(white) \\  
        *(white) & *(white) & *(white) & *(white) & *(white)
    \end{ytableau}}
    \quad \quad \quad \quad \quad \quad \quad \quad \quad \quad \quad \quad \quad \quad \quad 
    \raisebox{.5\height}{\begin{ytableau}
        *(white) \\ 
        *(white) \\  
        *(white) \\ 
        *(white) & *(white) & *(white) \\ 
        *(white) & *(white) & *(white) & *(white)
    \end{ytableau}}
\end{equation*}
\caption{The Young diagram of $\la = (5,2,2,1) \in \Y_4$ (left), and that of its conjugate partition $\la' = (4,3,1,1,1) \in \Y_5$ obtained by flipping the diagram across the diagonal. In general $\la_i' := \#\{j: \la_j \geq i\}$.}\label{fig:ferrers}
\end{figure}

Though this is certainly not trivial to see \emph{a priori}, these $p$-ranks all grow logarithmically:
\begin{equation}\label{eq:log_plus_fluc}
\rank(p^{i-1}G_n) \approx \log_p k(n) + \text{(fluctuations)} \quad \quad \quad \quad \text{as $n \to \infty$,}
\end{equation}
provided the number of products $k(n)$ does not grow too fast (for instance, \eqref{eq:log_plus_fluc} clearly cannot hold if $k(n) \gg p^n$ since the ranks are bounded by $n$). However, the much more interesting finding is that the fluctuations after centering are of order $O(1)$ and can be determined explicitly. Even the order of growth in \eqref{eq:log_plus_fluc} was not previously known in a universal setting, but in this work we are able to show universality even at the level of the exact fluctuations, finding the same random variables discovered in \cite{van2023local} in exactly-solvable cases.

 Hence we would like to show 
\begin{equation}\label{eq:expos_cvg_intro}
\rank(G_n) - \log_p k(n) \to X \quad \quad \quad \quad \text{in distribution as $n \to \infty$}
\end{equation}
for some discrete $\Z$-valued random variable $X$, and similarly for the other $\rank(p^{i-1}G_n)$ and their joint distribution. The limiting joint distribution of $(\rank(p^{i-1}G_n))_{1 \leq i \leq d}$ after centering should thus live on the set of \emph{integer signatures}
\begin{equation}\label{eq:defsigd_intro}
\Sig_d := \{(\la_1,\ldots,\la_d) \in \Z^d: \la_1 \geq \ldots \geq \la_d\}.
\end{equation}
For later we define the \emph{size} of a signature
\begin{equation}\label{eq:sigsize_intro}
|\la| := \sum_{i=1}^d \la_i,
\end{equation}
and note also that the set of nonnegative signatures 
\begin{equation}
\Sig_d^{\geq 0} := \{(\la_1,\ldots,\la_d) \in \Z^d: \la_1 \geq \ldots \geq \la_d \geq 0\}
\end{equation}
is just $\Y_d$.

Unless $k(n)$ is chosen very carefully, there is a trivial obstruction to the convergence \eqref{eq:expos_cvg_intro}, which is that $\rank(G_n)$ is an integer while $\log_p k(n)$ is in general not. This is not just an issue with this example, but a general issue with convergence of finite fluctuations of $\rank(G_n)$ for other $G_n$, in fact with convergence of recentered $\Z$-valued random variables. We treat it in \Cref{thm:matrix_product_intro} by requiring that the shift $-\log_p k(n)$ converges\footnote{Such convergence in $\R/\Z$ is implied by convergence of fractional parts $\{\log_p k(n)\}=\log_p k(n) - \floor{\log_p k(n)}$ to the lift of $\bar{\zeta}$ lying in $[0,1)$. However, convergence in $\R/\Z$ is a more uniform statement and slightly more general, e.g. the sequence $.9,1.1,.99,1.01,\ldots$ converges in $\R/\Z$ but its fractional parts do not converge.} in $\R/\Z$ to some $\bar{\zeta} \in \R/\Z$, as this shift may then simply be subtracted off to obtain a $\Z$-valued random variable. This convergence modulo $\Z$ can always be achieved by passing to a subsequence of $n$. Passing to subsequences is not just a technical restriction, and one actually obtains a one-parameter family of related but nontrivially distinct limit random variables for different subsequences. 

\begin{defi}\label{def:cL_intro}
For $p$ prime and $G$ an abelian $p$-group, we write $n_{max}(G)$ to be the number of chains of proper subgroups $0 = H_0 \lneq H_1 \lneq \ldots \lneq H_k = G$ of maximal length $k=\log_p |G|$. For $d \in \Z_{\geq 1}$, and $\chi \in \R_{>0}$, we define $\cL_{d,p^{-1},\chi}$ to be the unique $\Sig_d$-valued random variable with moments given by
\begin{equation}\label{eq:intro_cL_moments}
\E[p^{\la \cdot \cL_{d,p^{-1},\chi}}] = \frac{((p-1)\chi)^{|\la|}}{|\la|!} n_{max}(G_{\la'})
\end{equation}
for every $\la \in \Y_d$.
\end{defi}
It is not at all obvious that \Cref{def:cL_intro} defines a unique random variable, but we show this is true in \Cref{thm:cL_the_same}. We also show it matches the quite different definition in \cite[Theorem 6.1]{van2023local}, which gave explicit formulas for the probabilities $\Pr(\cL_{d,p^{-1},\chi} = \mu)$. As an example, for the random variable $\cL_{1,p^{-1},\chi}$ which governs the limiting fluctuations of $\rank(G_n) = \corank(A^{(n)}_{k(n)} \cdots A^{(n)}_1 \pmod{p})$, the formula reads 
\begin{equation*}
\Pr(\cL_{1,p^{-1},\chi} = x) = \frac{1}{\prod_{i \geq 1} (1-p^{-i})} \sum_{m \geq 0} e^{-\chi p^{m-x}} \frac{(-1)^m p^{-\binom{m}{2}}}{\prod_{j=1}^m (1-p^{-j})} \quad \quad \quad \quad \text{ for any $x \in \Z$.}
\end{equation*}
It is worth noting that $n_{max}$ is a rational function of $p$ and both \Cref{def:cL_intro} and the different definition \cite[Theorem 6.1]{van2023local} actually make sense and are equivalent for any real $p>1$, but in this work we will always assume $p$ is prime.

As $d$ grows the explicit formulas for $\cL_{d,p^{-1},\chi}$ are much more complicated than the formulas for the moments in \Cref{def:cL_intro}. An interesting feature of $\cL_{1,p^{-1},\chi}$ is that as $\chi$ varies over $[1,p)$ the real-valued random variables $\chi^{-1}p^{\cL_{1,p^{-1},\chi}}$ form a family of solutions to the indeterminate Stieltjes moment problem 
\begin{equation}
\E[X^m] = \frac{\prod_{i=1}^m (p^i-1)}{m!}, \quad \quad \quad \quad m=0,1,2,\ldots,
\end{equation}
with disjoint discrete supports $\chi^{-1} p^\Z$; this is a simple consequence of \eqref{eq:intro_cL_moments}. 

We may now state the theorem. We will always denote the fractional part of a real number $x$ by $\{x\} = x-\floor{x}$, and we write
\begin{equation}
\near{x} = \operatorname{argmin}_{y \in \Z} |x-y|
\end{equation}
for the nearest integer to $x$. It does not matter what convention is chosen for half-integers, since we will only apply $\near{\cdot}$ to sequences whose terms are getting closer and closer to integers (in other words, converging to $0$ in $\R/\Z$).

\begin{theorem}\label{thm:matrix_product_intro} 
Fix $p$ prime, let $\xi$ be a $\Z$-valued random variable such that $\xi \pmod{p}$ is nonconstant. For each $n \in \mathbb{Z}_{\geq 1}$ let $A_{i}^{(n)}, i \geq 1$ be iid $n \times n$ matrices over $\Z$ with iid $\xi$ entries. Let $(k(n))_{n \geq 1}$ be a sequence of natural numbers such that $k(n) \to \infty$ as $n \rightarrow \infty$ and $k(n) = O(e^{(\log n)^{1-\eps}})$ for some $0<\eps<1$. Define
\begin{equation}
G_n = \cok(A_{k(n)}^{(n)} \cdots A_{1}^{(n)})[p^\infty],
\end{equation}
and let $(n_j)_{j \geq 1}$ be any subsequence for which $-\log _{p} k(n_{j})$ converges in $\R/\Z$ to some $\bar{\zeta}$ as $j \to \infty$, and let $\zeta \in \R$ be any lift of $\bar{\zeta}$. Then for any $d \in \Z_{\geq 1}$,
\begin{equation}\label{eq:rmt_bulk_intro}
(\rank(p^{i-1}G_{n_j})-\near{\log _{p}(k(n_{j}))+\zeta})_{1 \leq i \leq d} \rightarrow \cL_{d,p^{-1}, p^{-\zeta}/(p-1)}
\end{equation}
in distribution as $j \rightarrow \infty$, where $\cL$ is as in \Cref{def:cL_intro}. Furthermore, the exact same result holds for matrices over the $p$-adic integers $\Z_p$.
\end{theorem}

\begin{rmk}
The definition of $\bar{\zeta}$ ensures that $\log _{p}(k(n_{j}))+\zeta$ is close to $\Z$ for large $j$, so the rounding $\near{\log _{p}(k(n_{j}))+\zeta}$ has little effect. Note also that the choice of lift $\zeta$ of $\bar{\zeta}$ changes both the shift on the left-hand side of \eqref{eq:rmt_bulk_intro}, and the parameter $p^{-\zeta}/(p-1)$ on the right-hand side; indeed, changing the $\chi$ parameter in \Cref{def:cL_intro} by an integer power of $p$ simply shifts $\cL_{d,p^{-1},\chi}$ by the same integer, as is easily checked via the moments. For an equivalent formulation which does not require the rounding $\near{\cdot}$ or the choice of lift, see \Cref{thm:noround_matrix_products}.
\end{rmk}

To our knowledge, \Cref{thm:matrix_product_intro} is the first universality result for matrix cokernels which do not converge to a fixed random group, and in particular the first one for $\cok(A_k \cdots A_1)$ where $k$ is not fixed. Furthermore, recalling that the sequence $\rank(p^{i-1}G_n), i = 1,\ldots,d$ determines $G_n/p^dG_n$, we have that for fixed $d$ the above completely determines the asymptotics of this group. The condition that the entry distribution be nonconstant modulo $p$ applies to very generic distributions including $0-1$ matrices, matrices with entries chosen uniformly from $[-b,b]$, and others, in addition to the additive Haar measure on $\Z_p$ considered in \cite{van2023local}. Though it establishes very fine information on the fluctuations of $\rank(p^{i-1}G_{n})$, it is quite different from previous results on convergence of random finite groups $(G_n)_{n \geq 1}$ to a random finite group $G$. 

Our motivation to establish such a result comes instead from complex random matrix theory and the interacting particle system perspective. For any nonsingular $A \in \Mat_n(\Z)$, the cokernel decomposes into cyclic factors
\begin{equation}
\cok(A)[p^\infty] \cong \bigoplus_{i=1}^n \Z/p^{\la_i}\Z
\end{equation}
for some nonnegative signature $\la = (\la_1,\ldots,\la_n)$. An equivalent description of $\la$ is that by Smith normal form, there always exist $U,V \in \GL_n(\Z)$ such that $UAV = \diag(a_1,\ldots,a_n)$ and $p^{\la_i}$ is the highest power of $p$ dividing $a_i$. This is analogous to singular value decomposition, and the numbers $\la_i$ are often called the \emph{singular numbers} of $A$ in the related $p$-adic setting. The present work as well as previous ones \cite{van2020limits,nguyen2022universality,van2023local,van2023reflecting} are partially inspired by the broad literature on singular values of matrix products, which begins with Bellman \cite{bellman1954limit} and Furstenberg-Kesten \cite{furstenberg1960products} in the 1950s and continues to recent works such as Akemann-Burda-Kieburg \cite{akemann2014universal,akemann2019integrable}, Crisanti-Paladin-Vulpiani \cite{crisanti2012products}, Liu-Wang-Wang \cite{liu2018lyapunov}, and others, see \cite[Appendix A]{van2023local} for a more thorough survey and comparison with the discrete case.

If one fixes $n$ and lets $k$ vary, the cokernel
\begin{equation}
\cok(A_k^{(n)} \cdots A_1^{(n)})[p^\infty] \cong \bigoplus_{i=1}^n \Z/p^{\la_i(k)}\Z
\end{equation}
yields a stochastic process $(\la_1(k),\ldots,\la_n(k))$ in discrete time $k=0,1,2,\ldots$. One can visualize $(\la_1(k),\ldots,\la_n(k))$ by a Young diagram as in \Cref{fig:ferrers}; as $k$ increases, this gives a growth process on such Young diagrams. 

The surprising observation of \cite{van2023reflecting} was that when $n$ is large, each of the singular numbers $\la_i(k), i \gg 1$ behaves as though it has a `clock' which rings at random intervals independent of all of the others. The singular number lies dormant until this `clock' rings, and then increases by $1$ unless $\la_i(k) = \la_{i-1}(k)$, in which case it is `blocked' by the next singular number, and certain simple interactions between them occur. These random dynamics by which the singular numbers $\la_i(k), i \gg 1$ evolve as $k$ increases are encapsulated in a continuous-time interacting particle system, the reflecting Poisson sea mentioned earlier, for which we refer to \cite[Sections 1.2 and 2]{van2023reflecting} for definitions. Its distribution at a single time $T$ was shown to be determined by marginals distributed as $\cL_{d,p^{-1},\chi}$ in \cite[Theorem 8.2]{van2023reflecting}. \Cref{thm:matrix_product_intro} thus shows that the convergence of matrix product cokernels to the reflecting Poisson sea is universal at the level of single-time marginals. 

\begin{rmk}
In particular, the $d=1$ case of \Cref{thm:matrix_product_intro} implies that when $k(n)$ grows, the matrix $A_{k(n)}^{(n)} \cdots A_1^{(n)}$ is getting closer to being singular, and that the corank of $A_{k(n)}^{(n)} \cdots A_1^{(n)} \pmod{p}$ is likely to grow relatively fast (with order $\log_p(k(n))$). In the complex case, when $A_{i}^{(n)}$ are independent matrices of iid standard gaussian, the results of Burda-Jarosz-Livan-Nowak-Swiech \cite{BJLNS} also support a similar situation that $A_{k(n)}^{(n)} \cdots A_1^{(n)}$ is close to being singular, by showing that the least singular value of $A_{k(n)}^{(n)} \cdots A_1^{(n)}$ is likely to decay to zero relatively fast (with order $n^{-k(n)/2}$). To the best of our knowledge, there has been no universality result in the literature concerning the spectrum of $A_{k}^{(n)} \cdots A_1^{(n)}$ in the complex setting for matrices with generic iid entries and $n,k \to \infty$, though for unitarily-invariant distributions such universality was established by Ahn \cite{ahn2022extremal}.
\end{rmk}

\subsection{Fluctuations of random abelian groups of growing size} The aforementioned works on the moment method for finite groups show that the so-called \emph{$H$-moments} $ \#\Sur(G_n,H)$ converge,
\begin{equation}
\lim_{n \to \infty} \E[\#\Sur(G_n,H)] = \E[\#\Sur(G,H)],
\end{equation}
for all $H$ from the appropriate class of groups. Under suitable conditions on the growth of $\E[\#\Sur(G,H)]$ as $H$ ranges, see for instance \cite[Theorems 8.2 and 8.3]{W0}, this is sufficient to conclude $G_n \to G$ in distribution. This is a very useful tool because asymptotics of $\E[\#\Sur(G_n,H)]$ can be computed when $G_n$ is the cokernel of a random matrix, for several broad classes of matrix distributions including products of random matrices with iid entries.

In our setting, $\lim_{n \to \infty} \E[\#\Sur(G_n,H)] = \infty$ for any nontrivial abelian $p$-group $H$, so this method of course does not apply. What is surprising, in our opinion, is that a similar method based on $H$-moments \emph{does} work. Namely, though the moments $\E[\#\Sur(G_n,H)]$ diverge as $n \to \infty$ and there is no limiting finite group $G$, we show that convergence of certain rescalings of these moments is enough to conclude convergence of the ranks $\rank(p^{i-1}G_n)$ to $\Sig_d$-valued random variables, as we describe now.

For our general results, it is better to work on a completion of $\Sig_d$ to avoid issues with escape of mass:
\begin{defi}\label{def:extended_sigs}
For any $d \in \N$ we define the set of \emph{extended integer signatures}
\begin{equation}
\bSig_d := \Big\{(\la_1,\ldots,\la_d) \in (\Z \cup \{-\infty\})^d: \la_1 \geq \ldots \geq \la_d\Big \}
\end{equation}
with the convention $a > -\infty$ for every $a \in \Z$. 

When speaking of weak convergence of measures on $\bSig_d$, we will always mean with respect to the topology whose open sets are generated by singleton sets $\{a\}$ and intervals $[-\infty,a] := \{-\infty\} \cup \Z_{\leq a}$ for $a \in \Z$; likewise, measures on $\bSig_d$ will always mean measures with respect to the Borel $\sigma$-algebra associated to the topology, which is just the discrete $\sigma$-algebra. 
\end{defi}

To state our moment growth condition, we will say that a function $\xi: \R \to \R$ has \emph{superlinear growth} (or simply is superlinear) if 
\begin{equation}
\lim_{x \to \infty} \xi(x) - \alpha x = \infty
\end{equation}
for every $\alpha \in \R_+$.

\begin{defi}\label{def:nicely_behaved}
For fixed constant $q > 1$ and $d \in \Z_{\geq 1}$, we say that a collection of constants $\{C_\la: \la \in \Sig_d^{\geq 0}\}$ is \emph{nicely-behaved} (with respect to $q$) if for any fixed integers $\la_2 \geq \ldots \geq \la_d \geq 0$, there exists $F > 0$ and $\xi$ of superlinear growth for which
\begin{equation}\label{eq:C_la_bound_t}
|C_{(\la_1,\ldots,\la_d)}| \leq F q^{\frac{1}{2}\la_1^2 - \xi(\la_1)}
\end{equation}
for all integers $\la_1 \geq \la_2$. If $d=1$, this should be interpreted as the condition that
\begin{equation}
|C_{(\la_1)}| \leq F q^{\frac{1}{2} \la_1^2-\xi(\la_1)}
\end{equation}
for all $\la_1 \in \Z_{\geq 0}$.
\end{defi}

The main result is the following.

\begin{theorem}[Rescaled moment method]\label{thm:group_convergence_intro}
Fix $p$ prime and $d \in \Z_{\geq 1}$. Let $(G_n)_{n \geq 1}$ be a sequence of random finitely-generated abelian $p$-groups and $(c_n)_{n \geq 1}$ a sequence of real numbers such that the following hold:
\begin{enumerate}[label=(\roman*)]
\item For every $\la \in \Sig_d^{\geq 0}$, $\E[\#\Hom(G_n,G_{\la'})]/p^{|\la|c_n}$ has a finite limit as $n \to \infty$,
\vskip .1in
\item The sequence $(-c_n)_{n \geq 1}$ converges in $\R/\Z$ to some $\bar{\zeta}$, and
\vskip .1in
\item For $\zeta \in \R$ any real lift of $\bar{\zeta}$, the collection $\{C_\la: \la \in \Sig_d^{\geq 0}\}$, defined by 
\begin{equation}\label{eq:renorm_moment_conv}
C_\la := p^{-\zeta |\la|} \lim_{n \to \infty} \frac{\E[\#\Hom(G_n,G_{\la'})]}{p^{|\la|c_n}},
\end{equation} 
is nicely-behaved with respect to $p$.
\end{enumerate}
Then there exists a unique $\bSig_d$-valued random variable $X = (X_1,\ldots,X_d)$ with moments $\E[p^{\sum_{i=1}^d \la_i X_i}] = C_\la$ for all $\la \in \Sig_d^{\geq 0}$, and the $\bSig_d$-valued random variables $X^{(n)}:=(\rank(p^{i-1}G_n) - \near{c_n + \zeta})_{1 \leq i \leq d}$ converge in distribution to $X$ as $n \to \infty$.
\end{theorem}


\begin{rmk}
In later work, Sawin-Wood \cite{sawin2022moment} showed how to extract the distribution of random algebraic structures, including groups\footnote{In the abelian group case, there is also an alternative proof from symmetric function theory in \cite{van2024symmetric}.}, from their moments. We give a partial analogue in \Cref{thm:dist_existence_intro} in the next section.
\end{rmk}

\begin{rmk}
In some cases such as the matrix product case of \Cref{thm:matrix_product_intro}, it is not actually necessary to pass to subsequences; one may instead show that the limiting cokernels are asymptotically well-approximated by $\cL_{d,p^{-1},\chi}$ for appropriate $\chi = \chi(n)$ as $n \to \infty$, see \cite[Theorem 1.2]{van2023local}. Such a statement can be deduced from the subsequence version in \Cref{thm:matrix_product_intro}, but this requires in addition the continuity of the probabilities in the parameter $\chi$. In that context this continuity follows from explicit formulas, but the analogue in the general setting of \Cref{thm:group_convergence_intro} is not so clear. Hence we stick to the subsequence formulation in this work.
\end{rmk}

\subsection{Methods} Since
\begin{equation*}
\#\Hom(G_{\mu'},G_{\la'}) = p^{\sum_i \mu_i \la_i}
\end{equation*}
(see e.g. \cite[Lemma 7.1]{W0}), the Hom-moment $\E[\#\Hom(G_n,G_{\la'})]$ is an exponential mixed moment in the usual sense of the random variables $\rank(p^{i-1}G_n),i=1,2,\ldots$. It might seem that one could obtain results such as \Cref{thm:group_convergence_intro} from classical probability, but in fact the condition that $\{C_\la: \la \in \Sig_d^{\geq 0}\}$ be nicely-behaved is much weaker than what is needed to ensure that these exponential moments determine a unique random variable, i.e. the classical moment problem is indeterminate. However, $\rank(p^{i-1}G_n)$ is always an integer, giving access to techniques which do not apply for arbitrary real-valued random variables. The basic technique we use was developed for $\rank(G_n)$ by Heath-Brown \cite{heath1993size,heath1994size} in the case when $\rank(G_n)$ converges without recentering to a random $\Z_{\geq 0}$-valued random variable, and extended by Wood \cite{W0} to convergence of $G_n$ to a random group $G$. Weaker versions of those results are recovered from \Cref{thm:group_convergence_intro} in the case when $c_n$ does not go to $\infty$, since then the convergence of $X^{(n)}$ above implies convergence of $G_n/p^d G_n$ to a random group. In this case, a weaker growth condition on the moments\footnote{Called `well-behaved' in \cite{sawin2022moment}, hence our terminology in \Cref{def:nicely_behaved}.} suffices, see \cite[Theorem 8.2]{W0}. Thus \Cref{thm:group_convergence_intro} naturally extends the moment method of \cite[Theorem 8.3]{W0} to the case when $\rank(G_n)$ diverges. Our proof of \Cref{thm:group_convergence_intro} is heavily inspired by \cite{W0}, but many additional complications arise when $\rank(G_n)$ diverges, see the discussion directly before \Cref{subsec:mom_conv_proof}.

To apply \Cref{thm:group_convergence_intro} to any specific sequence $(G_n)_{n \geq 1}$, the scaling constants $c_n$ must be chosen appropriately to obtain a meaningful statement: if $c_n \gg \rank(G_n)$ then the limits $C_\la$ in \eqref{eq:renorm_moment_conv} will always be $0$ except for $\la = (0,\ldots,0)$, and the conclusion of the theorem will simply be that $X^{(n)}$ converges weakly to the point mass $\delta_{(-\infty,\ldots,-\infty)}$. One must take $c_n \approx \E[\rank(G_n)]$: to obtain \Cref{thm:matrix_product_intro} for example, one must take $c_n = \log_p k(n)$. After doing so, the convergence in distribution stated in \Cref{thm:matrix_product_intro} is reduced to a statement about the limiting normalized Hom-moments\footnote{We remark that the Hom-moments $\E[\#\Hom(G_n,G_{\la'})]$ are always given by finite linear combinations of Sur-moments $\E[\#\Sur(G_n,H)]$ as $H$ ranges over subgroups of $G_{\la'}$; provided $\rank(p^{d-1}G_n) \to \infty$ as $n \to \infty$, the proportion of homomorphisms to a fixed $p^d$-torsion group which are surjections always goes to $1$, so one can replace $\Hom$ by $\Sur$ in \Cref{thm:group_convergence_intro}. The reason we state it this way is so that it uniformly covers the case when $\rank(p^{d-1}G_n) \not \to \infty$.} $\E[\#\Hom(G_n,G_{\la'})]$. This is nontrivial, but luckily for fixed $k$ the asymptotics of the moments of $G_n = \cok(A_{k}^{(n)} \cdots A_{1}^{(n)})[p^\infty]$ were computed in \cite{nguyen2022universality}, and so we are able to modify those computations to keep track of the dependence on $k$. More precisely, as the dependence over $k$ is crucial to us, in Section \ref{sect:support} we will reintroduce the key ingredients from \cite{W0,W1} with explicit constants. Later in Section \ref{section:prod:moments} we will exploit these tools by introducing two key parameter sequences $\eps_{k,n}, \eps_{k,n}'$ (see Proposition \ref{prop:single:k}) to control the growth rates in terms of $k$.

\subsection{Other growing random groups?} Though our only application of \Cref{thm:group_convergence_intro} in this paper is to cokernels of random matrix products, random groups of diverging size occur in natural examples beyond these. We conclude by mentioning a few.

\begin{example}
Recent work of M\'esz\'aros \cite{meszaros2024phase} treats $n \times n$ banded matrices $A^{(n)}$ with iid entries on the central band of width $w_n$ and all other entries $0$; the iid entries are taken in the $p$-adic integers, though the same proofs work over the integers. In particular, \cite[Theorem 5]{meszaros2024phase} shows that if $w_n-\log_p n \to -\infty$, then the random variables $\corank(A^{(n)} \pmod{p})$ do not form a tight sequence, i.e. there is escape of mass. No results on the growth or limiting fluctuations of these random variables are currently known; to the best of our knowledge \cite{meszaros2024phase} is the first work to treat banded matrices in this discrete setting.
\end{example}

\begin{example}
Let $A^{(n)}$ be the adjacency matrix of an Erd\"os-R\'enyi random graph on $n$ vertices, in the critical regime where each edge is taken with probability $\sim 1/n$. In this case, it is shown by Glasgow-Kwan-Sah-Sawhney \cite{GKSS,GKSS'} that the corank of $A^{(n)}$ over $\R$ has order $\text{const} \cdot n$, and its fluctuations are Gaussian. The literature on random groups has never treated the critical regime to our knowledge, however: we are not even aware of any work on $\corank(A^{(n)} \pmod{p})$, much less for $\cok(A^{(n)})[p^\infty]$, though it is believed \cite{sawhney2024} that $\corank(A^{(n)} \pmod{p})$ is also order $n$ and also has Gaussian fluctuations. This matrix $A^{(n)}$ is just a symmetric $0-1$ matrix with sparse entry distribution, and one may ask the same question for nonsymmetric sparse matrices, alternating sparse matrices, etc. in the critical regime where entries are nonzero with probability $\sim 1/n$. Bl{\"o}mer-Karp-Welzl \cite{blomer1997rank} and Coja-Oghlan-Erg{\"u}r-Gao-Hetterich-Rolvien \cite{coja2020rank} have also studied other sparse matrix distributions over $\F_p$ and shown the corank goes to $\infty$ with $n$ in their regimes, but without treating the fluctuations. 
\end{example}  

Just as the limiting cokernel distributions of non-sparse matrices depend on the symmetry class but are universal within symmetry classes, it is natural to wonder to what extent the cokernel fluctuations of random matrix models with diverging corank (modulo $p$) depend on the matrix distribution. We are not aware of any works apart from the present one which treat the full cokernel for matrix distributions where the growth of the corank (modulo $p$) is unbounded, but hope that \Cref{thm:group_convergence_intro} and the ideas behind it can help begin to answer some of these basic questions in the future.

\subsection{Outline} In \Cref{sec:moment_method} we prove \Cref{thm:group_convergence_intro}, along with another result (\Cref{thm:dist_existence_intro}) which extracts the probabilities of a $\Sig_d$-valued random variable from its moments; we do not use the latter in our random matrix context, but it is a basic result which may be useful in the future. In \Cref{sect:support} we record various supporting lemmas for a single matrix from \cite{W0}. In \Cref{section:prod:moments} we establish asymptotic control of the moments of matrix products, and in \Cref{sect:n_k} we combine these ingredients to prove \Cref{thm:matrix_product_intro}.

\section{The rescaled moment method for groups and the moment method for random signatures}\label{sec:moment_method}

The goal of this section is to prove \Cref{thm:group_convergence_intro} and \Cref{thm:dist_existence_intro}. Our proof of \Cref{thm:group_convergence_intro} also mostly takes place in the setting of random extended signatures: we prove the more general result \Cref{thm:moment_convergence_general} where the prime $p$ is replaced by a real parameter $q$, and only specialize to groups at the last possible moment. We expect these more general statements may be useful elsewhere: for instance, we remark that one may use \Cref{thm:moment_convergence_general} to give an independent proof of a certain result \cite[Theorem 10.2]{van2023local} regarding convergence of an interacting particle system, which is defined without reference to groups and features a parameter analogous to $p$ which does not have to be prime.

The more general result is the following, and for the sake of exposition we explain immediately below how it implies the statement \Cref{thm:group_convergence_intro} regarding Hom-moments of abelian $p$-groups.

\begin{theorem}\label{thm:moment_convergence_general}
Fix $q \in \R_{>1}$ and $d \in \Z_{\geq 1}$, and let $\{C_\la: \la \in \Sig_d^{\geq 0}\}$ be nicely-behaved with respect to $q$ (see \Cref{def:nicely_behaved}). Let $(X^{(n)})_{n \geq 1}$ be any sequence of $\bSig_d$-valued random variables such that 
\begin{equation}\label{eq:limit_moments_general}
\lim_{n \to \infty} \E[q^{\sum_{i=1}^d \la_i X_i^{(n)}}] = C_\la 
\end{equation}
for every $\la \in \Sig_d^{\geq 0}$. Then there exists a unique $\bSig_d$-valued random variable $X = (X_1,\ldots,X_d)$ with moments
\begin{equation}
\E[q^{\sum_{i=1}^d \la_i X_i}] = C_\la,
\end{equation}
and $X^{(n)} \to X$ in distribution as $n \to \infty$.
\end{theorem}

\begin{proof}[Proof of \Cref{thm:group_convergence_intro} from \Cref{thm:moment_convergence_general}]
Since 
\begin{equation}
\#\Hom(G,G_{\la'}) = p^{\sum_{i=1}^d \la_i \cdot \rank(p^{i-1}G)}
\end{equation}
for $\la \in \Sig_d^{\geq 0}$, the hypothesis \eqref{eq:renorm_moment_conv} yields that 
\begin{equation}\label{eq:almost_x_moments}
\lim_{n \to \infty} \E[p^{\sum_{i=1}^d \la_i \cdot (\rank(p^{i-1}G_n) - c_n - \zeta)}] = C_\la. 
\end{equation}
By hypothesis, 
\begin{equation}\label{eq:int_frac_part}
\lim_{n \to \infty} c_n + \zeta - \near{c_n+\zeta} = 0.
\end{equation}
Hence letting
\begin{equation}
X^{(n)} := (\rank(p^{i-1}G_n) - \near{c_n+\zeta})_{1 \leq i \leq d}
\end{equation}
as in the theorem statement, \eqref{eq:almost_x_moments} and \eqref{eq:int_frac_part} together imply
\begin{equation}
\lim_{n \to \infty} \E[p^{\sum_{i=1}^d \la_i X^{(n)}_i}] = C_\la. 
\end{equation}
Now \Cref{thm:moment_convergence_general} implies that $X^{(n)} \to X$ in distribution, where $X$ is the unique $\bSig_d$-valued random variable with 
\begin{equation*}
\E[p^{\sum_{i=1}^d \la_i X_i}] = C_\la. \qedhere
\end{equation*}
\end{proof}

We also prove explicit formulas in this setting which allow one to extract the weights and the `cumulative distribution function' with respect to the dominance partial order, which we now define.

\begin{defi}\label{def:dominance}
Given $\mu,\nu \in \bSig_d$, the \emph{dominance order} $\leq$ is the partial order defined by 
\begin{equation}
\mu \leq \nu \iff \sum_{j=1}^i \mu_j \leq \sum_{j=1}^i \nu_j \text{ for all }1 \leq i \leq d,
\end{equation}
where as usual we take the convention $-\infty + a = -\infty < a$ for every $a \in \Z$ when working with extended signatures.
\end{defi}

In what follows we frequently use the $q$-Pochhammer symbol defined as
$$(a;q)_k = \prod_{i=0}^{k-1}(1-aq^i),$$
for $k \geq 0$, with the convention $(a;q)_0=1$ in the case where the product is empty.

\begin{theorem}\label{thm:dist_existence_intro}
Let $q \in \R_{>1}$, $d \in \Z_{\geq 1}$, and $\meas$ be a probability measure on $\bSig_d$ such that the moments
\begin{equation}
\sum_{\mu \in \bSig_d} \meas(\{\mu\}) q^{\sum_{i=1}^d \la_i \mu_i} =: C_\la
\end{equation}
exist and $\{C_\la: \la \in \Sig_d^{\geq 0}\}$ is nicely-behaved with respect to $q$. Then it is the unique such measure, and its weights and `cumulative distribution function' on $\nu \in \Sig_d$ are given by
\begin{equation*}
\meas(\{\nu\}) = \sum_{\substack{\beta \in \Sig_d: \\ \beta \leq \nu}} \prod_{i=1}^d \frac{(-1)^{\sum_{j=1}^i \nu_j - \beta_j} q^{-\binom{\sum_{j=1}^i \nu_j - \beta_j}{2}}}{(q^{-1};q^{-1})_{\sum_{j=1}^i \nu_j - \beta_j}(q^{-1};q^{-1})_{\beta_{i-1}-\beta_i-\sum_{j=1}^i \nu_j - \beta_j}} \sum_{\substack{\la \in \Sig_d^{\geq 0}: \\ \la_2 \leq \beta_1-\beta_d}} a_\beta(\la) C_\la,
\end{equation*}
and 
\begin{multline*}
\meas(\{\mu \in \bSig_d: \mu \leq \nu \}) \\ = \sum_{\substack{\beta \in \Sig_d: \\ \beta \leq \nu}} \prod_{i=1}^d \frac{(-1)^{\sum_{j=1}^i \nu_j - \beta_j} q^{-\binom{\sum_{j=1}^i \nu_j - \beta_j+1}{2}}}{(q^{-1};q^{-1})_{\sum_{j=1}^i \nu_j - \beta_j}(q^{-1};q^{-1})_{\beta_{i-1}-\beta_i-\sum_{j=1}^i \nu_j - \beta_j}} \sum_{\substack{\la \in \Sig_d^{\geq 0}: \\ \la_2 \leq \beta_1-\beta_d}} a_\beta(\la) C_\la
\end{multline*}
where $a_\beta(\la)$ is as defined in \Cref{thm:H_from_melanie}. 
\end{theorem}

\begin{rmk}
Note that the formulas in \Cref{thm:dist_existence_intro} do not make sense for $\nu \in \bSig_d \setminus \Sig_d$.
\end{rmk}

\begin{rmk}
Let us also explain why it is important to work over $\bSig_d$ rather than $\Sig_d$. Even in the case $d=1$, the Dirac measure $\delta_{(n)}(\cdot)$ on $\Sig_1 = \Z$ has moments which converge (to $0$) as $n \to -\infty$, but the measure itself converges weakly (with respect to the discrete topology) to the zero measure. With respect to the topology of \Cref{def:extended_sigs}, however, it converges weakly to $\delta_{(-\infty)}(\cdot)$, which is still a probability measure. Naively, one might thus try to write a robustness theorem where the limit measure may have total mass less than $1$, which is no big issue. However, one still runs into problems for $d \geq 2$. For instance, the measure $\delta_{(0,n)}(\cdot)$ converges weakly (with respect to the discrete topology) to the zero measure, but the marginal distribution of the first coordinate is $\delta_{(0)}$. Hence the operations of taking marginals and taking limits do not commute, which seems not ideal. This is repaired by considering $\bSig_2$, as then the weak limit is $\delta_{(0,-\infty)}(\cdot)$, and this seems to be the best way to treat the escape of mass issue.
\end{rmk}

\begin{rmk}
It is somewhat surprising that the formulas for the weights and the CDF are almost identical, differing only in one power of $q$; this takes its origin in the elementary result \Cref{thm:int_ind_and_cdf}. Note that in \cite{sawin2022moment}, they also prove existence of probability measures provided that the sums defining $\meas(\{\nu\})$ are nonnegative. The reason we do not state a result of this type is that our formulas are only valid for $\nu \in \Sig_d$, and while one can use these to extract formulas for general $\nu \in \bSig_d$ (this is implicitly done in \Cref{thm:observables_suffice}), they are more complicated.
\end{rmk}

The strategy to prove \Cref{thm:moment_convergence_general} is the following. The $\la$-moment
\begin{equation}\label{eq:C_system}
C_\la = \sum_{\mu \in \bSig_d} \Pr(X = \mu) q^{\sum_{i=1}^d \la_i \mu_i}
\end{equation}
is an infinite linear combination of probabilities $\Pr(X = \mu)$. So, one wants to explicitly invert this system and write each $\Pr(X=\mu)$ in terms of the $C_\la$, in such a way that the coefficients are provably small enough that no analytic difficulties arise. The method of \cite{W0} is to introduce a family of functions $\hH_{d,q,\beta}(\mu)$ of $\mu \in \bSig_d$ which are indexed by $\beta \in \Sig_d$. These are certain infinite linear combinations of functions $\mu \mapsto q^{\sum_{i=1}^d \la_i \mu_i}$, and their key property is that $\hH_{d,q,\beta}(\mu)$ is nonzero if and only if $\mu \leq \beta$ in the dominance order\footnote{\cite[Lemma 8.1]{W0} proves a slightly weaker statement in terms of the lexicographic order, which is implied by this one. We mention also that the notation $\hH$ is not used in \cite{W0}, only the function $H$ of \Cref{thm:H_from_melanie}.}. In other words, these functions triangularize the system of equations \eqref{eq:C_system}. 

If the law of $X$ were supported on nonnegative signatures, as it is in \cite{W0} and other moment method works, the expectations of these functions $\hH_{d,q,\beta}$ are enough to determine the probabilities by taking finite linear combinations, and the coefficients featured in those linear combinations do not matter. However, in our case this is no longer true: the linear combinations are infinite, and so issues with interchanging sums arise. To deal with this, we explicitly invert the system, and write the indicators $\bbone(\mu=\nu)$ as an explicit infinite linear combination of $\hH_{d,q,\beta}(\mu)$ as $\beta$ ranges (\Cref{thm:prob_from_observables}). It is not at all obvious that the functions $\hH_{d,q,\beta}(\mu)$ of \cite{W0} would be explicitly invertible in this manner, nor that the coefficients appearing in the inversion would be small enough not to create analytic difficulties. Thankfully this is true---provided one assumes slightly more stringent moment growth conditions than in \cite{W0}, as we do in \Cref{def:nicely_behaved}, c.f. \cite[Theorem 8.2]{W0}. 

The fact that the entries in $\bSig_d$ are not bounded below creates difficulties elsewhere in the proof. Extra work (\Cref{thm:observables_suffice}) is required after proving the explicit formulas in \Cref{thm:prob_from_observables}, since these are only valid on $\Sig_d$ rather than $\bSig_d$. To prove \Cref{thm:moment_convergence_general}, it is also not enough to extract the probabilities $\Pr(X=\mu)$, as one must commute the limit as $n \to \infty$ through this inversion to obtain $\Pr(X^{(n)}=\mu) \to \Pr(X=\mu)$ from convergence of moments of $X^{(n)}$ to $X$. An argument of this type is given in the proof of \cite[Theorem 8.3]{W0}, but it breaks down once one considers $\bSig_d$ instead of $\Sig_d^{\geq 0}$; in order to repair it, one must introduce a cutoff parameter $c$ and divide the space $\bSig_d$ into subsets in a nontrivial manner. We now proceed to the proof.

\subsection{Proof of \Cref{thm:moment_convergence_general}.}\label{subsec:mom_conv_proof}

\begin{lemma}\label{thm:H_from_melanie} 
Given a positive integer $d$, a real $q > 1$, and $\beta \in \Sig_d$, define the power series
\begin{align}\label{eq:H_def_formula}
\begin{split}
&H_{d,q,\beta}(z_1,\ldots,z_d) := \\ 
&\prod_{j \geq \beta_1+1} \left(1-\frac{z_1}{q^j}\right)\prod_{j=\beta_{1}+\beta_{2}+1}^{2 \beta_{1}}\left(1-\frac{z_{2}}{q^{j}}\right) \prod_{j=\beta_{1}+\beta_{2}+\beta_{3}+1}^{\beta_{1}+2 \beta_{2}}\left(1-\frac{z_{3}}{q^{j}}\right) \cdots \prod_{j=\beta_{1} + \cdots+\beta_{d}+1}^{\beta_{1}+\cdots+\beta_{d-2}+2 \beta_{d-1}}\left(1-\frac{z_{d}}{q^{j}}\right) \\ 
&\quad \quad \quad \quad = \sum_{\substack{\la \in \Sig_d^{\geq 0} \\ \la_2 \leq \beta_1-\beta_d}} a_\beta(\la) z_1^{\la_1-\la_2} z_2^{\la_2-\la_3} \cdots z_d^{\la_d}
\end{split}
\end{align}
in $z_1,\ldots,z_d$. Then the Taylor coefficients of $H_{d,q,\beta}$ satisfy the bound 
\begin{equation}
\label{eq:a_la_bound}
\left|a_\beta(\la)\right| \leq E q^{-\beta_{1} (\la_1-\la_2) - \binom{\la_1-\la_2+1}{2}},
\end{equation}
for a constant $E$ depending on $d, q$, and $\beta$. 

\end{lemma}
\begin{proof}
The bound \eqref{eq:a_la_bound} is shown in \cite[Lemma 8.1]{W0} (the statement there assumes $q$ is prime, but the proof does not use this).
\end{proof}

\begin{lemma}\label{thm:H_bound_and_triangularity}
Let $d$ be a positive integer, $\beta \in \Sig_d$, and $q \in \R_{>1}$ as before. For $\mu \in \bSig_d$ set 
\begin{align}\label{eq:Hqmu}
\begin{split}
\hH_{d,q,\beta}(\mu) &:=  H_{d,q,\beta}(q^{\mu_{1}}, q^{\mu_{1}+\mu_{2}},\ldots, q^{\mu_{1}+\cdots+\mu_{d}}) \\ 
&= (q^{-(\beta_1-\mu_1+1)};q^{-1})_\infty (q^{-(\beta_1+\beta_2-\mu_1-\mu_2+1)};q^{-1})_{\beta_1-\beta_2} \cdots (q^{-(|\beta|-|\mu|+1)};q^{-1})_{\beta_{d-1}-\beta_d}
\end{split}
\end{align}
where $H$ is as in \Cref{thm:H_from_melanie}. Then for all $\mu \in \bSig_d$, 
\begin{equation}\label{eq:H_bound}
0 \leq \hH_{d,q,\beta}(\mu) \leq 1
\end{equation}
and
\begin{equation}\label{eq:triangularity}
\text{$\hH_{d,q,\beta}(\mu)$ is nonzero if and only if $\mu \leq \beta$}
\end{equation}
where $\leq$ is the dominance order.
\end{lemma}
\begin{proof}
For the forward direction of \eqref{eq:triangularity}, suppose $\hH_{d,q,\beta}(\mu) \neq 0$. Then each of the $q$-Pochhammer symbols in \eqref{eq:Hqmu} is nonzero. The nonvanishing of the first implies that $\mu_1 \leq \beta_1$. The nonvanishing of the second implies that either $\mu_1+\mu_2 \leq \beta_1+\beta_2$, or $\mu_1+\mu_2 > 2\beta_1$. However, $\mu_1+\mu_2 \leq 2\mu_1 \leq 2\beta_1$, so the second cannot occur, hence $\mu_1+\mu_2 \leq \beta_1+\beta_2$. Continuing like this, the nonvanishing of all $q$-Pochhammer symbols implies that $\sum_{j=1}^i \mu_j \leq \sum_{j=1}^i \beta_j$, i.e. $\mu \leq \beta$. Note that with our convention $-\infty+a=-\infty$, these steps are still valid if $\mu$ has some parts $-\infty$.

For the backward direction, suppose $\mu \leq \beta$. The $i\tth$ $q$-Pochhammer symbol in \eqref{eq:Hqmu} is nonzero if $\sum_{j=1}^i \beta_j - \mu_j \geq 0$, hence all are nonzero, so $\hH_{d,q,\beta}(\mu) \neq 0$.

To establish \eqref{eq:H_bound}, note that we have already established it when $\mu \not \leq \beta$ since the expression is $0$. This leaves the case $\mu \leq \beta$, and in this case $\sum_{j=1}^i \mu_j \leq \sum_{j=1}^i \beta_j$ for each $i$, from which it is easy to see that each $q$-Pochhammer symbol in \eqref{eq:Hqmu} lies between $0$ and $1$.
\end{proof}

The triangularity property \eqref{eq:triangularity} is a slightly stronger version of a similar statement in \cite[Lemma 8.1]{W0}. It implies that in principle, for any fixed $\nu \in \Sig_d$ one may write the indicator $\mu \mapsto \bbone(\mu=\nu)$ as some infinite linear combination of the functions $\mu \mapsto \hH_{d,q,\beta}(\mu)$ on $\Sig_d$ for different $\beta$. \Cref{thm:H_invert} gives this linear combinations explicitly, and \Cref{thm:H_to_cdf} gives a different formula for the function $\mu \mapsto \bbone(\mu \leq \nu)$. Each of these results is essentially just an extension to signatures of the following result for integers, which is an application of the $q$-binomial theorem.

\begin{lemma}\label{thm:int_ind_and_cdf}
For any $n \in \Z,m\in \bZ$,
\begin{equation}
\label{eq:int_H_invert}
\frac{1}{(q^{-1};q^{-1})_\infty^2} \sum_{b \in \Z} (-1)^{n-b}q^{-\binom{n-b}{2}} (q^{-(n-b+1)};q^{-1})_\infty (q^{-(b-m+1)};q^{-1})_\infty = \bbone(n=m)
\end{equation}
and 
\begin{equation}
\label{eq:int_H_cdf}
\frac{1}{(q^{-1};q^{-1})_\infty^2} \sum_{b \in \Z} (-1)^{n-b}q^{-\binom{n-b+1}{2}} (q^{-(n-b+1)};q^{-1})_\infty (q^{-(b-m+1)};q^{-1})_\infty = \bbone(m \leq n)
\end{equation}
where we take the usual convention $q^{-\infty}:=0$, $-\infty < k$ for each $k \in \Z$, and $-\infty+k=-\infty$.
\end{lemma}

\begin{proof}
First note that for $k \in \Z$,
\begin{equation}
(q^{-k};q^{-1})_\infty = \begin{cases} 0 & k \leq 0 \\ 
\frac{(q^{-1};q^{-1})_\infty}{(q^{-1};q^{-1})_{k-1}} & k > 0.
\end{cases}
\end{equation}
Hence the only nonzero terms in the sum in \eqref{eq:int_H_invert} comes when $m \leq b \leq n$, so if $n < m$ the sum is $0$ and equality holds. If $n \geq m$, the sum becomes 
\begin{equation}
\sum_{\substack{b \in \Z \\ m \leq  b \leq n}} \frac{(-1)^{n-b}q^{-\binom{n-b}{2}}}{(q^{-1};q^{-1})_{n-b}(q^{-1};q^{-1})_{b-m}}  = \frac{1}{(q^{-1};q^{-1})_{n-m}} (1;q^{-1})_{n-m} = \bbone(n=m)
\end{equation}
by the $q$-binomial theorem (this formula still makes sense when $m=-\infty$ with our conventions).

Similarly to \eqref{eq:int_H_invert}, the left hand side of \eqref{eq:int_H_cdf} has all terms $0$ if $n < m$, which checks that case of the equality. If $n \geq m$ it is 
\begin{equation}
\sum_{\substack{b \in \Z \\ m \leq  b \leq n}} \frac{(-1)^{n-b}q^{-\binom{n-b+1}{2}}}{(q^{-1};q^{-1})_{n-b}(q^{-1};q^{-1})_{b-m}}  = \frac{1}{(q^{-1};q^{-1})_{n-m}} (q^{-1};q^{-1})_{n-m} = 1,
\end{equation}
again by the $q$-binomial theorem, and again this makes sense when $m=-\infty$.
\end{proof}

The next two results, \Cref{thm:H_invert} and \Cref{thm:H_to_cdf}, upgrade the two parts of \Cref{thm:int_ind_and_cdf} from integers to signatures. We remark that each part of \Cref{thm:int_ind_and_cdf} can be proven from the other by a calculation (the right hand sides of \eqref{eq:int_H_invert} and \eqref{eq:int_H_cdf} are related by summation/M\"obius inversion, and one just has to calculate how this changes the coefficients on the left hand sides). The same is true of deriving the next two results from one another, but we find it easier to prove them independently as above.

\begin{lemma}\label{thm:H_invert}
For any $\nu \in \Sig_d, \mu \in \bSig_d$ and $\hH_{d,q,\beta}$ as in \Cref{thm:H_bound_and_triangularity}, we have
\begin{equation}
\label{eq:H_invert}
\sum_{\substack{\beta \in \Sig_d: \\ \beta \leq \nu}} \prod_{i=1}^d \frac{(-1)^{\sum_{j=1}^i \nu_j - \beta_j} q^{-\binom{\sum_{j=1}^i \nu_j - \beta_j}{2}}(q^{-(\beta_{i-1}-\beta_i-\sum_{j=1}^i \nu_j - \beta_j)-1};q^{-1})_\infty}{(q^{-1};q^{-1})_\infty (q^{-1};q^{-1})_{\sum_{j=1}^i \nu_j - \beta_j}}\hH_{d,q,\beta}(\mu) = \bbone(\mu=\nu),
\end{equation}
where for the $i=1$ term in the product we take the convention that $\beta_0 + \nu_1 = \infty$ for any $\nu_1 \in \Z$. Here we note that $\mu$ may be an extended signature while $\nu$ may not, and adopt the usual convention $q^{-\infty} := 0$.
\end{lemma}
\begin{proof}
The case $d=1$ is exactly \eqref{eq:int_H_invert}, taking $\nu = (n)$ and $\mu = (m)$. For the general case we induct using the same argument. We claim that for any $d \geq 2$,
\begin{multline}
\label{eq:induct_H_invert}
\text{LHS\eqref{eq:H_invert}} = \sum_{\substack{\tbeta \in \Sig_{d-1}: \\ \tbeta \leq (\nu_1,\ldots,\nu_{d-1}) \\ \tbeta \geq (\mu_1,\ldots,\mu_{d-1})}} \prod_{i=1}^{d-1} \frac{(-1)^{\sum_{j=1}^i \nu_j - \tbeta_j} q^{-\binom{\sum_{j=1}^i \nu_j - \tbeta_j}{2}}(q^{-(\tbeta_{i-1}-\tbeta_i-\sum_{j=1}^i \nu_j - \tbeta_j)-1};q^{-1})_\infty}{(q^{-1};q^{-1})_\infty (q^{-1};q^{-1})_{\sum_{j=1}^i \nu_j - \tbeta_j}} \\ 
\times \hH_{d-1,q,\tbeta}((\mu_1,\ldots,\mu_{d-1}))\bbone(\sum_{j=1}^d \mu_j = \sum_{j=1}^d \nu_j).
\end{multline}
First note that by the triangularity property of $\hH_{d,q,\beta}$, the sum on the left hand side of \eqref{eq:H_invert} may be taken to be over $\beta$ for which $\beta \geq \mu$, as all other terms are $0$. This in particular implies that
\begin{equation}\label{eq:betasize}
|\mu| \leq |\beta| \leq |\nu|
\end{equation}
for all $\beta$ which are summed over. Additionally, the summands on the left hand side of \eqref{eq:H_invert} are bounded above by $(q^{-1};q^{-1})_\infty^{-2d}q^{-\binom{\sum_{j=1}^i \nu_j - \beta_j}{2}}$ by \eqref{eq:H_bound} in \Cref{thm:H_bound_and_triangularity} and the trivial bound $(q^{-1};q^{-1})_n \geq (q^{-1};q^{-1})_\infty$. Hence the sum converges absolutely, so we may reorder terms however desired. Splitting the sum over $\beta$ into a sum over $(\beta_1,\ldots,\beta_{d-1}) =:\tbeta$ and $\beta_d =:b$ and taking into account the restriction on $b$ given by \eqref{eq:betasize}, we have
\begin{multline}\label{eq:split_beta_sum}
\text{LHS\eqref{eq:H_invert}} = \sum_{\substack{\tbeta \in \Sig_{d-1}: \\\tbeta \leq (\nu_1,\ldots,\nu_{d-1})  \\ \tbeta \geq (\mu_1,\ldots,\mu_{d-1})}} \prod_{i=1}^{d-1} \frac{(-1)^{\sum_{j=1}^i \nu_j - \tbeta_j} q^{-\binom{\sum_{j=1}^i \nu_j - \tbeta_j}{2}}(q^{-(\tbeta_{i-1}-\tbeta_i-\sum_{j=1}^i \nu_j - \tbeta_j)-1};q^{-1})_\infty}{(q^{-1};q^{-1})_\infty (q^{-1};q^{-1})_{\sum_{j=1}^i \nu_j - \tbeta_j}}  \\ 
\times  \hH_{d-1,q,\tbeta}((\mu_1,\ldots,\mu_{d-1})) \sum_{\substack{b \in \Z: \\ b \leq |\nu|-|\tbeta|, \\ b \geq |\mu| - |\tbeta|} } F_{\tbeta}(b),
\end{multline}
where 
\begin{equation*}
  F_{\tbeta}(b)  =
 \frac{(-1)^{\nu_d-b+\sum_{j=1}^{d-1} \nu_j - \tbeta_j}  (q^{-(\tbeta_{d-1}-\nu_d-\sum_{j=1}^{d-1} \nu_j - \tbeta_j)-1};q^{-1})_\infty}{q^{\binom{\nu_d-b+\sum_{j=1}^{d-1} \nu_j - \tbeta_j}{2}}(q^{-1};q^{-1})_\infty (q^{-1};q^{-1})_{\nu_d - b + \sum_{j=1}^{d-1} \nu_j - \tbeta_j}} (q^{|\mu|-|\tbeta|-b-1};q^{-1})_{\tbeta_{d-1}-b}.
\end{equation*}
Since $\tbeta+b - |\mu| \geq 0$, we may write
\begin{equation}
(q^{|\mu|-|\tbeta|-b-1};q^{-1})_{\tbeta_{d-1}-b} = \frac{(q^{-1};q^{-1})_{-|\mu|+|\tbeta|+\tbeta_{d-1}}}{(q^{-1};q^{-1})_{-|\mu|+|\tbeta|+b}}.
\end{equation}
Now the same $q$-binomial trick as for the $d=1$ case yields for the second sum in \eqref{eq:split_beta_sum} that
\begin{align*}\label{eq:use_qbinom_trick}
\begin{split}
\sum_{b = |\mu|-|\tbeta|}^{|\nu|-|\tbeta|} F_{\tbeta}(b) &= \frac{(q^{-1};q^{-1})_{-|\mu|+|\tbeta|+\tbeta_{d-1}}(q^{-(\tbeta_{d-1}-\nu_d-\sum_{j=1}^{d-1} \nu_j - \tbeta_j)-1};q^{-1})_\infty}{(q^{-1};q^{-1})_\infty} \\ 
&\quad \quad \quad \quad \quad \quad \quad \quad \quad \times \sum_{b = |\mu|-|\tbeta|}^{|\nu|-|\tbeta|}  \frac{(-1)^{|\nu| - |\tbeta| - b} q^{-\binom{|\nu|-|\tbeta| - b}{2}}}{(q^{-1};q^{-1})_{b-(|\mu|-|\tbeta|)}(q^{-1};q^{-1})_{|\nu| - |\tbeta| - b} } \\ 
  &= \frac{(q^{-1};q^{-1})_{-|\mu|+|\tbeta|+\tbeta_{d-1}}(q^{-(\tbeta_{d-1}-\nu_d-\sum_{j=1}^{d-1} \nu_j - \tbeta_j)-1};q^{-1})_\infty}{(q^{-1};q^{-1})_\infty} \bbone(\sum_{j=1}^d \mu_j = \sum_{j=1}^d \nu_j).
\end{split}
\end{align*}
When $|\mu|=|\nu|$, 
\begin{equation*}
    (q^{-1};q^{-1})_{-|\mu|+|\tbeta|+\tbeta_{d-1}}(q^{-(\tbeta_{d-1}-\nu_d-\sum_{j=1}^{d-1} \nu_j - \tbeta_j)-1};q^{-1})_\infty = (q^{-1};q^{-1})_\infty,
  \end{equation*} 
so the above simplifies to $\bbone(\sum_{j=1}^d \mu_j = \sum_{j=1}^d \nu_j)$. This shows \eqref{eq:induct_H_invert}. 

Applying \eqref{eq:induct_H_invert} iteratively yields that
\begin{equation}
\text{RHS\eqref{eq:H_invert}} = \prod_{i=1}^d \bbone(\sum_{j=1}^i \mu_j = \sum_{j=1}^i \nu_j) = \bbone(\mu=\nu),
\end{equation}
completing the proof.
\end{proof}

The next two results, \Cref{thm:prob_from_observables} and \Cref{thm:observables_suffice}, show that probabilities may be extracted from moments. The reason we have broken this into two results is that the probability of a signature $\nu \in \Sig_d$ has an explicit expression in terms of moments given in \Cref{thm:prob_from_observables}, while for an extended signature $\nu \in \bSig_d \setminus \Sig_d$ this formula does not hold, and we simply show that the moments determine the probability without giving an explicit formula (\Cref{thm:observables_suffice}). 

\begin{prop}\label{thm:prob_from_observables}
Fix a positive integer $d$ and real $q>1$. Let $\{x_\mu: \mu \in \bSig_d\}$ be a set of nonnegative real numbers such that the sums
\begin{equation}
C_\la := \sum_{\mu \in \bSig_d} x_\mu q^{\sum_{i=1}^d \la_i \mu_i}
\end{equation}
are convergent for every $\la \in \Sig_d^{\geq 0}$, and $\{C_\la: \la \in \Sig_d^{\geq 0}\}$ is nicely-behaved with respect to $q$. Then for any $\nu \in \Sig_d$,
\begin{equation}\label{eq:prob_from_observables}
x_\nu = \sum_{\substack{\beta \in \Sig_d: \\ \beta \leq \nu}} \prod_{i=1}^d \frac{(-1)^{\sum_{j=1}^i \nu_j - \beta_j} q^{-\binom{\sum_{j=1}^i \nu_j - \beta_j}{2}}}{(q^{-1};q^{-1})_{\sum_{j=1}^i \nu_j - \beta_j}(q^{-1};q^{-1})_{\beta_{i-1}-\beta_i-\sum_{j=1}^i \nu_j - \beta_j}} \sum_{\substack{\la \in \Sig_d^{\geq 0}: \\ \la_2 \leq \beta_1-\beta_d}} a_\beta(\la) C_\la
\end{equation}
where $a_\beta(\la)$ is as defined in \Cref{thm:H_from_melanie}. Note \eqref{eq:prob_from_observables} is not valid in general for $\nu \in \bSig_d$.
\end{prop}

\begin{proof}[Proof of \Cref{thm:prob_from_observables}]
First consider a fixed $\beta \in \Sig_d$ with $\beta \leq \nu$. With $a_\beta(\la)$ defined as in \Cref{thm:H_from_melanie}, we have
\begin{equation}\label{eq:H_expand}
\hH_{d,q,\beta}(\mu) = \sum_{\substack{\la \in \Sig_d^{\geq 0}: \\ \la_2 \leq \beta_1-\beta_d}} a_\beta(\la) q^{\sum_{i=1}^d \la_i \mu_i}.
\end{equation}
We wish to show that the sum
\begin{equation}\label{eq:x_mu_abs_conv0}
\sum_{\mu \in \bSig_d}  \sum_{\substack{\la \in \Sig_d^{\geq 0}: \\ \la_2 \leq \beta_1-\beta_d}} a_\beta(\la) q^{\sum_{i=1}^d \la_i \mu_i} x_\mu
\end{equation}
converges absolutely. First, by rewriting as 
\begin{equation}\label{eq:x_mu_abs_conv}
\sum_{\mu \in \bSig_d}  \sum_{\substack{(\la_2,\ldots,\la_d) \in \Y_{d-1}: \\ \la_2 \leq \beta_1-\beta_d}} \sum_{\la_1 \geq \la_2} a_\beta(\la) q^{\sum_{i=1}^d \la_i \mu_i} x_\mu,
\end{equation}
the sum over $\la_2,\ldots,\la_d$ is finite and may be commuted past the sum over $\mu$, so it suffices to show that for all $\beta \in \Sig_d$ and all choices of $\la_2,\ldots,\la_d$ summed over in \eqref{eq:x_mu_abs_conv}, the sum
\begin{equation}\label{eq:abs_cvg_la1}
\sum_{\mu \in \bSig_d} \sum_{\substack{\la_1: \\ \la_1 \geq \la_2}} a_\beta(\la) q^{\sum_{i=1}^d \la_i \mu_i} x_\mu
\end{equation}
converges absolutely. By the Fubini-Tonelli theorem and nonnegativity of the $x_\mu$,
\begin{align}
  \begin{split}
    \sum_{\mu \in \bSig_d} \sum_{\substack{\la_1: \\ \la_1 \geq \la_2}} |a_\beta(\la) q^{\sum_{i=1}^d \la_i \mu_i} x_\mu| &=  \sum_{\substack{\la_1: \\ \la_1 \geq \la_2}} |a_\beta(\la)| \sum_{\mu \in \bSig_d}q^{\sum_{i=1}^d \la_i \mu_i} x_\mu \\ 
    &= \sum_{\substack{\la_1: \\ \la_1 \geq \la_2}} |a_\beta(\la)| C_\la.
  \end{split}
\end{align}
Hence by the bound \Cref{thm:H_from_melanie} on $|a_\beta(\la)|$, it suffices to show 
\begin{equation}
\label{eq:abs_sum_less_infty}
\sum_{\substack{\la_1: \\ \la_1 \geq \la_2}} E q^{-\beta_{1} (\la_1-\la_2)-\binom{\la_1-\la_2+1}{2}} |C_\la| < \infty,
\end{equation}
which follows straightforwardly from the nicely-behaved hypothesis (see \Cref{def:nicely_behaved}). Hence the sum \eqref{eq:x_mu_abs_conv0} converges absolutely, so by Fubini's theorem the sums over $\mu$ and $\la$ may be exchanged. This together with \eqref{eq:H_expand} and the definition of $C_\la$ yields
\begin{equation}\label{eq:H_and_C}
\sum_{\mu \in \bSig_d}  x_\mu \hH_{d,q,\beta}(\mu) = \sum_{\mu \in \bSig_d}  x_\mu \sum_{\substack{\la \in \Sig_d^{\geq 0}: \\ \la_2 \leq \beta_1-\beta_d}} a_\beta(\la) q^{\sum_{i=1}^d \la_i \mu_i} = \sum_{\substack{\la \in \Sig_d^{\geq 0}: \\ \la_2 \leq \beta_1-\beta_d}} a_\beta(\la) C_\la.
\end{equation}

We now consider the sum over $\beta$, rather than fixing $\beta$ as above, and show \eqref{eq:prob_from_observables}. By \eqref{eq:H_bound} in \Cref{thm:H_bound_and_triangularity}, and the trivial bound
\begin{equation}
\frac{1}{(q^{-1};q^{-1})_k} \leq \frac{1}{(q^{-1};q^{-1})_\infty},
\end{equation}
we have
\begin{align}\label{eq:bound_xmu_sum}
\begin{split}
&\sum_{\substack{\beta \in \Sig_d: \\ \beta \leq \nu}}\sum_{\mu \in \bSig_d}  \abs*{ \prod_{i=1}^d \frac{(-1)^{\sum_{j=1}^i \nu_j - \beta_j} q^{-\binom{\sum_{j=1}^i \nu_j - \beta_j}{2}}}{(q^{-1};q^{-1})_{\sum_{j=1}^i \nu_j - \beta_j}(q^{-1};q^{-1})_{\beta_{i-1}-\beta_i-\sum_{j=1}^i \nu_j - \beta_j}}  \hH_{d,q,\beta}(\mu)x_\mu} \\ 
&\leq \sum_{\substack{\beta \in \Sig_d: \\ \beta \leq \nu}}\sum_{\mu \in \bSig_d} \frac{1}{(q^{-1};q^{-1})_\infty^{2d}} \prod_{i=1}^d q^{-\binom{\sum_{j=1}^i \nu_j - \beta_j}{2}} x_\mu  \\ 
& = \frac{1}{(q^{-1};q^{-1})_\infty^{2d}} \left(\sum_{\substack{\beta \in \Sig_d: \\ \beta \leq \nu}} q^{-\sum_{i=1}^d \binom{\sum_{j=1}^i \nu_j - \beta_j}{2}}\right) \left(\sum_{\mu \in \bSig_d} x_\mu\right) < \infty
\end{split}
\end{align}
since the finiteness of $C_{(0,\ldots,0)}$ implies summability of the $x_\mu$. Hence by Fubini's theorem, the two sums in \eqref{eq:prob_from_observables} can be interchanged, so
\begin{multline}
\label{eq:H_sum_swap}
\text{RHS\eqref{eq:prob_from_observables}} = \sum_{\mu \in \bSig_d} x_\mu \sum_{\substack{\beta \in \Sig_d: \\ \beta \leq \nu}} \prod_{i=1}^d \frac{(-1)^{\sum_{j=1}^i \nu_j - \beta_j} q^{-\binom{\sum_{j=1}^i \nu_j - \beta_j}{2}}}{(q^{-1};q^{-1})_{\sum_{j=1}^i \nu_j - \beta_j}(q^{-1};q^{-1})_{\beta_{i-1}-\beta_i+\sum_{j=1}^i \nu_j - \beta_j}} \\ \times H_{d,q,\beta}(q^{\mu_1},\ldots,q^{\mu_1+\ldots+\mu_d}).
\end{multline}
The proof of \eqref{eq:prob_from_observables} now follows from \eqref{eq:H_sum_swap} together with \Cref{thm:H_invert}. 
\end{proof}

\begin{prop}\label{thm:observables_suffice}
Fix a positive integer $d$ and real $q>1$. Let $\{x_\mu: \mu \in \bSig_d\}$ and $\{y_\mu: \mu \in \bSig_d\}$ be two sets of nonnegative real numbers such that for each $\la \in \Sig_d^{\geq 0}$, the sums 
\begin{equation}\label{eq:x_y_moments}
\sum_{\mu \in \bSig_d} x_\mu q^{\sum_{i=1}^d \la_i \mu_i} = \sum_{\mu \in \bSig_d} y_\mu q^{\sum_{i=1}^d \la_i \mu_i} =: C_\la
\end{equation}
converge and are equal, and the collection $\{C_\la: \la \in \Sig_d^{\geq 0}\}$ is nicely-behaved with respect to $q$ in the sense of \Cref{def:nicely_behaved}. Then $x_\mu=y_\mu$ for all $\mu \in \bSig_d$.
\end{prop}
\begin{proof}

We show by induction on $\ell$ that for each $\ell = 0, \ldots,d$ and $(\nu_1,\ldots,\nu_{d-\ell}) \in \Sig_{d-\ell}$,
\begin{equation}
x_{(\nu_1,\ldots,\nu_{d-\ell},-\infty,\ldots,-\infty)} = y_{(\nu_1,\ldots,\nu_{d-\ell},-\infty,\ldots,-\infty)}.
\end{equation}
The base case $\ell=0$ follows directly from \Cref{thm:prob_from_observables}, so assume $1 \leq \ell \leq d$. We first claim that for any $(\nu_1,\ldots,\nu_{d-\ell}) \in \Sig_{d-\ell}$,
\begin{equation}\label{eq:trunc_marg_x_y}
x_{(\nu_1,\ldots,\nu_{d-\ell})} := \sum_{-\infty \leq \nu_d \leq \ldots \leq \nu_{d-\ell+1}} x_{(\nu_1,\ldots,\nu_d)} = \sum_{-\infty \leq \nu_d \leq \ldots \leq \nu_{d-\ell+1}} y_{(\nu_1,\ldots,\nu_d)} =: y_{(\nu_1,\ldots,\nu_{d-\ell})}. 
\end{equation}
By taking $\la$ in \eqref{eq:x_y_moments} with $\la_{d-\ell+1} = 0$, we have
\begin{equation}
\sum_{(\nu_1,\ldots,\nu_{d-\ell}) \in \bSig_{d-\ell}} x_{(\nu_1,\ldots,\nu_{d-\ell})} q^{\sum_{i=1}^{d-\ell} \la_i \nu_i} = \sum_{\nu \in \bSig_d} x_\nu q^{\sum_{i=1}^{d-\ell} \la_i \nu_i},
\end{equation}
hence the equality
\begin{equation}
\sum_{(\nu_1,\ldots,\nu_{d-\ell}) \in \bSig_{d-\ell}} x_{(\nu_1,\ldots,\nu_{d-\ell})} q^{\sum_{i=1}^{d-\ell} \la_i \nu_i} = \sum_{(\nu_1,\ldots,\nu_{d-\ell}) \in \bSig_{d-\ell}} y_{(\nu_1,\ldots,\nu_{d-\ell})} q^{\sum_{i=1}^{d-\ell} \la_i \nu_i}
\end{equation}
is a special case of \eqref{eq:x_y_moments}. By convergence of the sum formula for $C_{(0,\ldots,0)}$, the $x_\mu$ and $y_\mu$ are summable, and it follows easily that the numbers $x_{(\nu_1,\ldots,\nu_{d-\ell})}$ and $y_{(\nu_1,\ldots,\nu_{d-\ell})}$ are summable, hence \eqref{eq:trunc_marg_x_y} follows by applying \Cref{thm:prob_from_observables} with the $d$ in that result given by our $d-\ell$; note that we have only shown \eqref{eq:trunc_marg_x_y} when $(\nu_1,\ldots,\nu_{d-\ell}) \in \Sig_{d-\ell}$, not $\bSig_{d-\ell}$. But now 
\begin{align}
\begin{split}
x_{(\nu_1,\ldots,\nu_{d-\ell},-\infty,\ldots,-\infty)} &= x_{(\nu_1,\ldots,\nu_{d-\ell})} - \sum_{\substack{-\infty \leq \nu_d \leq \ldots \leq \nu_{d-\ell+1} \\ \nu_{d-\ell+1} > -\infty}} x_{(\nu_1,\ldots,\nu_d)} \\ 
&= y_{(\nu_1,\ldots,\nu_{d-\ell})} - \sum_{\substack{-\infty \leq \nu_d \leq \ldots \leq \nu_{d-\ell+1} \\ \nu_{d-\ell+1} > -\infty}} y_{(\nu_1,\ldots,\nu_d)} \\ 
&=y_{(\nu_1,\ldots,\nu_{d-\ell},-\infty,\ldots,-\infty)},
\end{split}
\end{align}
where we applied \eqref{eq:trunc_marg_x_y} and the inductive hypothesis in $\ell$ for the middle equality. This completes the induction and hence the proof.
\end{proof}

The below lemma is precisely our reason for working on the compactification $\bSig_d$ rather than $\Sig_d$, as it gives us tightness without having to show lower tail bounds for $X^{(n)}$ (upper tail bounds follow from convergence of Hom-moments, as we see in the lemma's proof).

\begin{lemma}\label{thm:tightness}
Let $\{C_\la: \la \in \Sig_d^{\geq 0}\}$ be any nonnegative real numbers, and let $(\meas_n)_{n \geq 1}$ be a sequence of probability measures satisfying \eqref{eq:limit_moments_general} for all $\la \in \Sig_d^{\geq 0}$. Then the sequence $(\meas_n)_{n \geq 1}$ is uniformly tight with respect to the topology of \Cref{def:extended_sigs}.
\end{lemma}
\begin{proof}
Suppose the conclusion is false. Then there exists some $\eps > 0$ for which there does not exist any compact set $K \subset \bSig_d$ with $\meas_n(K) > 1-\eps$ for all $n$. In particular, for any $u \in \Z$ the set 
\begin{equation}
K_u := \{\la \in \bSig_d: \la_1 \leq u\}
\end{equation}
is compact, hence for any $u$ 
\begin{equation}\label{eq:limsup_prob_eps}
\limsup_{n \to \infty} \meas_n(\bSig_d \setminus K_u) \geq \eps.
\end{equation}
By \eqref{eq:limit_moments_general} with $\la = (1,0,\ldots,0)$, we have that $\sum_{\mu \in \bSig_d} q^{\mu_1} \meas_n(\mu)$ converges as $n \to \infty$, hence it is in particular bounded above by some $D$. Hence by the Chernoff bound,
\begin{equation}
\meas_n(\{\mu \in \Sig_d: \mu_1 \geq a\}) \leq \frac{D}{q^{a}}
\end{equation}
for all $n$. Choosing $a \in \Z$ large enough that $D/q^a < \eps$ and setting $u=a-1$ yields a contradiction with \eqref{eq:limsup_prob_eps}, completing the proof.
\end{proof}

\begin{proof}[Proof of \Cref{thm:moment_convergence_general}]
First, by Prokhorov's theorem and \Cref{thm:tightness} we may choose a subsequence of $\N$ along which $X^{(n)}$ converges weakly in distribution to some probability measure $\tmeas$ on $\bSig_d$. For notational convenience we let $\tX$ be a $\bSig_d$-valued random variable with law $\tmeas$. Below, we use $\lim'_{n\to\infty}$ to denote the limit as $n \to \infty$ along this chosen subsequence. 

We wish to show, for any $\la \in \Sig_d^{\geq 0}$, that 
\begin{equation}
\sum_{\mu \in \bSig_d} \Pr(\tX=\mu) q^{\sum_{i=1}^d \la_i \mu_i} = C_\la,
\end{equation}
for which it suffices to show
\begin{equation}
\label{eq:apply_dct}
\lim'_{n \to \infty} \sum_{\mu \in \bSig_d} \Pr(X^{(n)}=\mu) q^{\sum_{i=1}^d \la_i \mu_i} = \sum_{\mu \in \bSig_d} \Pr(\tX=\mu) q^{\sum_{i=1}^d \la_i \mu_i},
\end{equation}
since the left hand side is $C_\la$ by \eqref{eq:limit_moments_general}. We assume without loss of generality that $\la_d > 0$, since if $\la = (\la_1,\ldots,\la_j,0,\ldots,0)$ then both sides of \eqref{eq:apply_dct} are independent of $X^{(n)}_{j+1},\ldots,X^{(n)}_d$ and of $\tX_{j+1},\ldots,\tX_d$ respectively, so we may give the same proof with $d$ replaced by $j$ and $X^{(n)}, \tX$ replaced by $(X^{(n)}_1,\ldots,X^{(n)}_j)$ and $(\tX_1,\ldots,\tX_j)$. We note that the assumption $\la_d > 0$ guarantees that the only contributions to the sum on the right hand side of \eqref{eq:apply_dct} come from $\mu \in \Sig_d$, which is what we need to use later; if $\la_d$ were equal to $0$, then terms with $\mu_d =-\infty$ would potentially contribute to the sum.

For the sake of exposition, we first show \eqref{eq:apply_dct} in the case $d=1$. For any cutoff parameter $c \in \Z$, we write 
\begin{equation}
\label{eq:c_split_d=1}
\text{LHS\eqref{eq:apply_dct}} = \lim'_{n \to \infty} \sum_{m \geq c} \Pr(X^{(n)} = (m)) q^{\la_1 m} + \sum_{m < c} \Pr(X^{(n)} = (m)) q^{\la_1 m}
\end{equation}
(the summand when $m=-\infty$ is $0$, see the previous paragraph). We first treat the first sum. Note that by hypothesis, the moment
\begin{equation}
\sum_{(m) \in \bSig_1} \Pr(X^{(n)} = (m)) q^{(2\la_1+1) m}
\end{equation}
converges as $n \to \infty$, hence in particular it is bounded above by some constant $D_\la$, and so since all summands are positive
\begin{equation}
\sum_{m \geq c} \Pr(X^{(n)} = (m)) q^{(2\la_1+1) m} \leq D_\la.
\end{equation}
Hence
\begin{align}\label{eq:exp_Dla_bound}
\begin{split}
\Pr(X^{(n)} = (m)) q^{\la_1 m} &= \frac{q^{\la_1 m}}{q^{(2\la_1+1)m}} \Pr(X^{(n)} = (m)) q^{(2\la_1+1)m}
\\
&\leq  q^{(-\la_1-1)m} \sum_{m' \in \bZ} \Pr(X^{(n)} = (m')) q^{(2\la_1+1)m'} \\
&\leq  q^{(-\la_1-1)m} D_\la.
\end{split}
\end{align}
Since $\sum_{m \geq c} q^{(-\la_1-1)m}$ converges, this shows that $m \mapsto \Pr(X^{(n)} = (m)) q^{\la_1 m}$ is dominated by an integrable function $m \mapsto q^{(-\la_1-1)m} D_\la$ on $\Z_{\geq c}$. Hence dominated convergence applied to the first sum yields
\begin{equation}\label{eq:d=1_c_limit}
\text{RHS\eqref{eq:c_split_d=1}} = \sum_{m \geq c} \Pr(\tX = (m)) q^{\la_1 m} + \lim'_{n \to \infty}\sum_{m < c} \Pr(X^{(n)} = (m)) q^{\la_1 m}.
\end{equation}
This argument of course does not work for the second sum, since the right hand side of \eqref{eq:exp_Dla_bound} is not summable over $m < c$. This is the reason for introducing the cutoff, as in lieu of dominated convergence we may simply send $c \to -\infty$ to kill the second sum. To do this, we note the trivial bound 
\begin{equation}\label{eq:c_bound_d=1}
\sum_{m < c} \Pr(X^{(n)} = (m)) q^{\la_1 m} \leq q^{c \la_1},
\end{equation}
since the left hand side is maximized when $X^{(n)} = c-1$ with probability $1$ (recall $\la_1 \geq 1$). Because $c$ does not appear in \eqref{eq:apply_dct}, we may insert $\lim_{c \to -\infty}$ to obtain
\begin{align}
\begin{split}
\text{LHS\eqref{eq:apply_dct}} &= \lim_{c \to -\infty} \text{LHS\eqref{eq:apply_dct}} \\ 
&=  \lim_{c \to -\infty} \sum_{m \geq c} \Pr(\tX = (m)) q^{\la_1 m} + \lim_{c \to -\infty} \lim'_{n \to \infty}\sum_{m < c} \Pr(X^{(n)} = (m)) q^{\la_1 m} \\ 
&= \sum_{m \in \Z} \Pr(\tX = (m)) q^{\la_1 m} + 0 \\ 
&= \sum_{m \in \bZ} \Pr(\tX = (m)) q^{\la_1 m}
\end{split}
\end{align}
where we used monotone convergence on the first summand, the bound \eqref{eq:c_bound_d=1} on the second sum, and then in the last line noted that $\Pr(\tX = (m)) q^{\la_1 m}=0$ when $m=-\infty$ because $\la_1 > 0$. This establishes \eqref{eq:apply_dct} when $d=1$. 

Now we prove \eqref{eq:apply_dct} for general $d$, which uses the same trick, with a more complicated version of the decomposition \eqref{eq:c_split_d=1}. For any cutoff parameter $c \in \Z$, we may split into $d+1$ sums via 
\begin{equation}\label{eq:c_split}
\text{LHS\eqref{eq:apply_dct}} = \lim'_{n \to \infty} \left(\sum_{\substack{\mu \in \bSig_d \\ \mu_d \geq c}} +  \sum_{\substack{\mu \in \bSig_d \\ \mu_d < c \leq \mu_{d-1}}} + \sum_{\substack{\mu \in \bSig_d \\ \mu_{d-1} < c \leq \mu_{d-2}}} + \ldots + \sum_{\substack{\mu \in \bSig_d \\ \mu_1 < c}} \right)(\Pr(X^{(n)} = \mu) q^{\sum_{i=1}^d \la_i \mu_i}).
\end{equation}
We first treat the first sum, which will be the only one to contribute in the limit when $c \to -\infty$ that we take later: we claim that
\begin{equation}
\label{eq:dct_lbd}
\lim'_{n \to \infty} \sum_{\substack{\mu \in \bSig_d \\ \mu_d \geq c}} \Pr(X^{(n)} = \mu) q^{\sum_{i=1}^d \la_i \mu_i} = \sum_{\substack{\mu \in \bSig_d \\ \mu_d \geq c}} \Pr(\tX = \mu) q^{\sum_{i=1}^d \la_i \mu_i},
\end{equation}
which we will show by dominated convergence.

To this end let $\tl = (2\la_i+1)_{1 \leq i \leq d} \in \Sig_d^{\geq 0}$, and note that 
\begin{equation}
\sum_{\substack{\mu \in \bSig_d \\ \mu_d \geq c}} \frac{q^{\sum_{i=1}^d \la_i \mu_i}}{q^{\sum_{i=1}^d \tl_i \mu_i}} = q^{c(|\la|-|\tl|)} \sum_{\mu_1 \geq \ldots \geq \mu_d \geq 0} q^{\sum_{i=1}^d \mu_i(-\la_i-1)}
\end{equation}
converges. There exists a constant $D_\la$ such that
\begin{equation}
 \sum_{\substack{\mu \in \bSig_d}} \Pr(X^{(n)} = \mu)  q^{\sum_{i=1}^d \tl_i \mu_i} < D_\la 
\end{equation}
for all $n$, since the left hand side converges as $n \to \infty$ by \eqref{eq:renorm_moment_conv}. Thus the summands in \eqref{eq:dct_lbd} are bounded by 
\begin{align}
\begin{split}
\Pr(X^{(n)} = \mu) q^{\sum_{i=1}^d \la_i \mu_i} &= \frac{q^{\sum_{i=1}^d \la_i \mu_i}}{q^{\sum_{i=1}^d \tl_i \mu_i}} \Pr(X^{(n)} = \mu) q^{\sum_{i=1}^d \tl_i \mu_i}
\\
&\leq \frac{q^{\sum_{i=1}^d \la_i \mu_i}}{q^{\sum_{i=1}^d \tl_i \mu_i}} \sum_{\nu \in \bSig_d} \Pr(X^{(n)} = \nu) q^{\sum_{i=1}^d \tl_i \nu_i} \\ 
& \leq \frac{q^{\sum_{i=1}^d \la_i \mu_i}}{q^{\sum_{i=1}^d \tl_i \mu_i}}D_\la.
\end{split}
\end{align}
Furthermore, the sum over $\{\mu \in \bSig_d: \mu_d \geq c\}$ of this expression converges as we just showed, so \eqref{eq:dct_lbd} follows by the Lebesgue dominated convergence theorem. 

Now we treat the other sums in \eqref{eq:c_split}, so let $0 \leq \ell \leq d-1$ and consider the sum 
\begin{multline}\label{eq:l_sum}
\sum_{\substack{\mu \in \bSig_d \\ \mu_{\ell+1} < c \leq \mu_{\ell}}} \Pr(X^{(n)} = \mu) q^{\sum_{i=1}^d \la_i \mu_i} = \sum_{\substack{\nu \in \bSig_{\ell} \\ \nu_{\ell} \geq c}} \Pr((X_1^{(n)},\ldots,X_{\ell}^{(n)})=\nu) q^{\sum_{i=1}^\ell \la_i \nu_i} \\ 
\times \sum_{\substack{\rho \in \bSig_{d-\ell} \\ \rho_1 < c}} \Pr((X_{\ell+1}^{(n)},\ldots,X_{d}^{(n)})=\rho|(X_1^{(n)},\ldots,X_{\ell}^{(n)})=\nu) q^{\sum_{i=1}^{d-\ell} \la_{i+\ell}\rho_i},
\end{multline}
where when $\ell=0$ we ignore the $c \leq \mu_{\ell}$ inequality (or take $\mu_0 :=\infty$). For the inner sum we have the bound
\begin{equation}
\sum_{\substack{\rho \in \bSig_{d-\ell} \\ \rho_1 < c}} \Pr((X_{\ell+1}^{(n)},\ldots,X_{d}^{(n)})=\rho|(X_1^{(n)},\ldots,X_{\ell}^{(n)})=\nu) q^{\sum_{i=1}^{d-\ell} \la_{i+\ell}\rho_i} \leq q^{c\sum_{i=1}^{d-\ell}\la_{i+\ell}}
\end{equation}
by bounding the $q^{\sum_{i=1}^{d-\ell} \la_{i+\ell}\rho_i}$ factors by the above and then noting the probabilities must sum to $1$. Hence 
\begin{equation}\label{eq:bound_c_sum}
\text{RHS\eqref{eq:l_sum}} \leq q^{c\sum_{i=1}^{d-\ell}\la_{i+\ell}} \sum_{\substack{\nu \in \bSig_{\ell} \\ \nu_{\ell} \geq c}} \Pr((X_1^{(n)},\ldots,X_{\ell}^{(n)})=\nu) q^{\sum_{i=1}^\ell \la_i \nu_i}.
\end{equation}
But 
\begin{multline}
  \sum_{\substack{\nu \in \bSig_{\ell} \\ \nu_{\ell} \geq c}} \Pr((X_1^{(n)},\ldots,X_{\ell}^{(n)})=\nu) q^{\sum_{i=1}^\ell \la_i \nu_i} \\ \leq \sum_{\substack{\nu \in \bSig_{\ell} }} \Pr((X_1^{(n)},\ldots,X_{\ell}^{(n)})=\nu) q^{\sum_{i=1}^\ell \la_i \nu_i}  \to \E[q^{\sum_{i=1}^\ell \la_i \nu_i}]
\end{multline}
by the hypothesis \eqref{eq:renorm_moment_conv} with the partition $(\la_1,\ldots,\la_\ell,0,\ldots,0)$. Hence there is a constant $C_\ell$ independent of $c$ such that 
\begin{equation}\label{eq:bound_c_sum2}
0 \leq \text{RHS\eqref{eq:l_sum}} \leq q^{c\sum_{i=1}^{d-\ell}\la_{i+\ell}} C_\ell
\end{equation}
for all $n$. Since the limit on the left hand side of 
\begin{multline}\label{eq:limit_split_c}
\lim'_{n \to \infty} \sum_{\mu \in \bSig_d} \Pr(X^{(n)}=\mu) q^{\sum_{i=1}^d \la_i \mu_i} = \lim'_{n \to \infty} \sum_{\substack{\mu \in \bSig_d \\ \mu_d \geq c}} \Pr(X^{(n)} = \mu) q^{\sum_{i=1}^d \la_i \mu_i}  \\ + \lim'_{n \to \infty} \sum_{\ell=0}^{d-1} \sum_{\substack{\mu \in \bSig_d \\ \mu_{\ell+1} < c \leq \mu_{\ell}}} \Pr(X^{(n)} = \mu) q^{\sum_{i=1}^d \la_i \mu_i}
\end{multline}
exists by \eqref{eq:renorm_moment_conv}, and the first limit on the right exists by \eqref{eq:dct_lbd}, the second limit on the right exists also. However, by \eqref{eq:bound_c_sum2},
\begin{equation}\label{eq:pc_bound}
\lim'_{n \to \infty} \sum_{\ell=0}^{d-1} \sum_{\substack{\mu \in \bSig_d \\ \mu_{\ell+1} < c \leq \mu_{\ell}}} \Pr(X^{(n)} = \mu) q^{\sum_{i=1}^d \la_i \mu_i} \leq \sum_{\ell=0}^{d-1} q^{c\sum_{i=1}^{d-\ell}\la_{i+\ell}} C_\ell.
\end{equation}
Because we assumed $\la_d > 0$, 
\begin{equation}\label{eq:cdl_pc_bd}
q^{c\sum_{i=1}^{d-\ell}\la_{i+\ell}} \leq q^c
\end{equation}
for all $0 \leq \ell \leq d-1$ and all $c \leq 0$. Because the left hand side of \eqref{eq:limit_split_c} is manifestly independent of $c$ (and hence so is the right hand side, less obviously),
\begin{align}\label{eq:limit_of_splits}
\begin{split}
\text{LHS\eqref{eq:limit_split_c}} &= \lim_{c \to -\infty} \text{RHS\eqref{eq:limit_split_c}} \\
&= \lim_{c \to -\infty} \sum_{\substack{\mu \in \bSig_d \\ \mu_d \geq c}} \Pr(\tX = \mu) q^{\sum_{i=1}^d \la_i \mu_i} + \lim_{c \to -\infty} \lim'_{n \to \infty} \sum_{\ell=0}^{d-1} \sum_{\substack{\mu \in \bSig_d \\ \mu_{\ell+1} < c \leq \mu_{\ell}}} \Pr(X^{(n)} = \mu) q^{\sum_{i=1}^d \la_i \mu_i} \\ 
&= \sum_{\substack{\mu \in \bSig_d}} \Pr(\tX = \mu) q^{\sum_{i=1}^d \la_i \mu_i} + 0
\end{split}
\end{align}
where we used \eqref{eq:dct_lbd} for the second line, and then used monotone convergence (all summands are clearly nonnegative) for the first limit in $c$, and the bound \eqref{eq:pc_bound} together with the observation \eqref{eq:cdl_pc_bd} for the second limit. Combining \eqref{eq:c_split} with \eqref{eq:limit_of_splits} yields \eqref{eq:apply_dct}. This, in particular, shows that there exists a probability measure ($\tmeas$) with $\la$-moments $C_\la$.


Now suppose for the sake of contradiction that 
\begin{equation}\label{eq:lim_for_diag}
\lim_{n \to \infty} \Pr(X^{(n)}=\mu)
\end{equation}
did not exist for at least some $\mu \in \bSig_d$. Then by tightness (\Cref{thm:tightness}) and a diagonalization argument, there exist two subsequences of $\N$ along which the limits \eqref{eq:lim_for_diag} exist for every $\mu$, but differ for at least one $\mu$. We know that the $\la$-moments of \eqref{eq:limit_moments_general} converge when the limit in $n$ in \eqref{eq:limit_moments_general} is replaced with the limit along one of our subsequences. However, we may apply the above argument to each subsequence and obtain convergence in distribution to another limit law $\tmeas'$, with $\tmeas'(\{\mu\})\neq \tmeas(\{\mu\})$. This contradicts \Cref{thm:observables_suffice}. Hence the limits \eqref{eq:lim_for_diag} exist for each $\mu \in \bSig_d$, completing the proof.
\end{proof}

\subsection{Proof of \Cref{thm:dist_existence_intro}.} This section mainly consists of CDF versions of \Cref{thm:H_invert} and \Cref{thm:prob_from_observables}, given as \Cref{thm:H_to_cdf} and \Cref{thm:cdf_from_observables} below. The proofs are the same as their counterparts, using the second part of \Cref{thm:int_ind_and_cdf} instead of the first.

\begin{lemma}\label{thm:H_to_cdf}
For any $\nu \in \Sig_d, \mu \in \bSig_d$ and $\hH_{d,q,\beta}$ as in \Cref{thm:H_bound_and_triangularity}, we have
\begin{multline}
\label{eq:H_to_cdf}
\sum_{\substack{\beta \in \Sig_d: \\ \beta \leq \nu}} \prod_{i=1}^d \frac{(-1)^{\sum_{j=1}^i \nu_j - \beta_j} q^{-\binom{(\sum_{j=1}^i \nu_j - \beta_j)+1}{2}}(q^{-(\beta_{i-1}-\beta_i-(\sum_{j=1}^i \nu_j - \beta_j))-1};q^{-1})_\infty}{(q^{-1};q^{-1})_\infty (q^{-1};q^{-1})_{\sum_{j=1}^i \nu_j - \beta_j}}\hH_{d,q,\beta}(\mu) \\ 
= \bbone(\mu \leq \nu),
\end{multline}
where for the $i=1$ term in the product we take the convention that $\beta_0 + \nu_1 = \infty$ for any $\nu_1 \in \Z$. Here we note that $\mu$ may be an extended signature while $\nu$ may not, and adopt the usual convention $q^{-\infty} := 0$.
\end{lemma}
\begin{proof}
The case when $d=1$ is exactly the second part of \Cref{thm:int_ind_and_cdf}, and as with \Cref{thm:H_invert}, the strategy is to use it inductively for general $d$. As in that proof, we may write 
\begin{multline}\label{eq:split_beta_sum_cdf}
\text{LHS\eqref{eq:H_to_cdf}} \\ 
= \sum_{\substack{\tbeta \in \Sig_{d-1}: \\\tbeta \leq (\nu_1,\ldots,\nu_{d-1})  \\ \tbeta \geq (\mu_1,\ldots,\mu_{d-1})}} \prod_{i=1}^{d-1} \frac{(-1)^{\sum_{j=1}^i \nu_j - \tbeta_j} q^{-\binom{(\sum_{j=1}^i \nu_j - \tbeta_j)+1}{2}}(q^{-(\tbeta_{i-1}-\tbeta_i-(\sum_{j=1}^i \nu_j - \tbeta_j))-1};q^{-1})_\infty}{(q^{-1};q^{-1})_\infty (q^{-1};q^{-1})_{\sum_{j=1}^i \nu_j - \tbeta_j}}  \\ 
\times \hH_{d-1,q,\tbeta}((\mu_1,\ldots,\mu_{d-1})) \\ 
\times \sum_{\substack{b \in \Z: \\ b \leq |\nu|-|\tbeta|, \\ b \geq |\mu| - |\tbeta|} }
 \frac{(-1)^{\nu_d-b+\sum_{j=1}^{d-1} \nu_j - \tbeta_j} q^{-\binom{\nu_d-b+(\sum_{j=1}^{d-1} \nu_j - \tbeta_j)+1}{2}} (q^{-(\tbeta_{d-1}-\nu_d-(\sum_{j=1}^{d-1} \nu_j - \tbeta_j))-1};q^{-1})_\infty}{(q^{-1};q^{-1})_\infty (q^{-1};q^{-1})_{\nu_d - b + \sum_{j=1}^{d-1} \nu_j - \tbeta_j}} \\ 
 \times (q^{|\mu|-|\tbeta|-b-1};q^{-1})_{\tbeta_{d-1}-b}.
\end{multline}
The second part of \Cref{thm:int_ind_and_cdf} yields that the second line of \eqref{eq:split_beta_sum_cdf} is $\bbone(\sum_{j=1}^d \mu_j \leq \sum_{j=1}^d \nu_j)$, and the argument can be continued inductively as with \Cref{thm:H_invert}.
\end{proof}

\begin{prop}\label{thm:cdf_from_observables}
In the same setup of \Cref{thm:prob_from_observables}, for any $\nu \in \Sig_d$
\begin{multline}\label{eq:cdf_from_observables}
\sum_{\mu \leq \nu} x_\mu \\ 
= \sum_{\substack{\beta \in \Sig_d: \\ \beta \leq \nu}} \prod_{i=1}^d \frac{(-1)^{\sum_{j=1}^i \nu_j - \beta_j} q^{-\binom{(\sum_{j=1}^i \nu_j - \beta_j)+1}{2}}}{(q^{-1};q^{-1})_{\sum_{j=1}^i \nu_j - \beta_j}(q^{-1};q^{-1})_{\beta_{i-1}-\beta_i-(\sum_{j=1}^i \nu_j - \beta_j)}}  \sum_{\substack{\la \in \Sig_d^{\geq 0}: \\ \la_2 \leq \beta_1-\beta_d}} a_\beta(\la) C_\la.
\end{multline}
\end{prop}
\begin{proof}[Proof of \Cref{thm:cdf_from_observables}]
As shown in the proof of \Cref{thm:prob_from_observables}, the sum
\begin{equation}\label{eq:x_mu_abs_conv2}
\sum_{\mu \in \bSig_d}  \sum_{\substack{\la \in \Sig_d^{\geq 0}: \\ \la_2 \leq \beta_1-\beta_d}} a_\beta(\la) q^{\sum_{i=1}^d \la_i \mu_i} x_\mu
\end{equation}
converges absolutely, and by the same Fubini argument as for \eqref{eq:bound_xmu_sum} (the only difference is the power of $q$) we have
\begin{equation*}
\text{RHS\eqref{eq:cdf_from_observables}} = \sum_{\mu \in \bSig_d} x_\mu \sum_{\substack{\beta \in \Sig_d: \\ \beta \leq \nu}} \prod_{i=1}^d \frac{(-1)^{\sum_{j=1}^i \nu_j - \beta_j} q^{-\binom{(\sum_{j=1}^i \nu_j - \beta_j)+1}{2}}}{(q^{-1};q^{-1})_{\sum_{j=1}^i \nu_j - \beta_j}(q^{-1};q^{-1})_{\beta_{i-1}-\beta_i+(\sum_{j=1}^i \nu_j - \beta_j)}}\hH_{d,q,\beta}(\mu).
\end{equation*} 
The proof now follows from \Cref{thm:H_to_cdf}.
\end{proof}

\begin{proof}[Proof of \Cref{thm:dist_existence_intro}]
Uniqueness follows from \Cref{thm:observables_suffice}, and the explicit formulas follow from \Cref{thm:prob_from_observables} and \Cref{thm:cdf_from_observables} respectively.
\end{proof}

\subsection{Extended remark: discrete limits on shifted lattices}

In this work, we have chosen to shift random signatures by integers, so that they remain random integer signatures. This is the perspective taken in \cite{van2023local,van2023reflecting} where the random variables $\cL_{d,p^{-1},\chi}$ have appeared previously. It is natural from the dynamical point of view of the reflecting Poisson sea given there, and it also has the advantage that the limit random variable and all prelimit random variables live on the same discrete state space $\bSig_d$. However, from other angles it is instead more natural to shift by real numbers instead and thereby obtain random variables valued on 
\begin{equation}
  \bSig_d + \zeta := \{(\mu_1+\zeta,\ldots,\mu_d+\zeta): (\mu_i)_{1 \leq i \leq d} \in \bSig_d\}.
\end{equation}
The advantage is that this removes the necessity for rounding and choosing lifts from $\R/\Z$ to $\R$. In this subsection, we outline this other perspective and why it is equivalent. This material is not necessary for any other results in the paper, but may be helpful for some readers.

The analogue of \Cref{thm:group_convergence_intro} is as follows.

\begin{theorem}\label{thm:noround_groups_general}
Fix $p$ prime and $d \in \Z_{\geq 1}$. Let $(G_n)_{n \geq 1}$ be a sequence of random finitely-generated abelian $p$-groups and $(c_n)_{n \geq 1}$ a sequence of real numbers such that the following hold:
\begin{enumerate}[label=(\roman*)]
\item For every $\la \in \Sig_d^{\geq 0}$, $\E[\#\Hom(G_n,G_{\la'})]/p^{|\la|c_n}$ has a finite limit as $n \to \infty$,
\vskip .1in
\item The collection $\{C_\la: \la \in \Sig_d^{\geq 0}\}$, defined by 
\begin{equation}\label{eq:renorm_moment_conv_nozeta}
C_\la := \lim_{n \to \infty} \frac{\E[\#\Hom(G_n,G_{\la'})]}{p^{|\la|c_n}},
\end{equation} 
is nicely-behaved with respect to $p$.
\vskip .1 in
\item The limit $\lim_{n \to \infty} -c_n =: \zeta$ exists in $\R/\Z$.
\end{enumerate}
Then there exists a unique $\bSig_d+\zeta$-valued random variable $X = (X_1,\ldots,X_d)$ with moments $\E[p^{\sum_{i=1}^d \la_i X_i}] = C_\la$ for all $\la \in \Sig_d^{\geq 0}$, and the $(\R \cup \{-\infty\})^d$-valued random variables $X^{(n)}:=(\rank(p^{i-1}G_n) - c_n)_{1 \leq i \leq d}$ converge weakly to $X$ as $n \to \infty$.
\end{theorem}

Note that by contrast to \eqref{eq:renorm_moment_conv} in \Cref{thm:group_convergence_intro}, there is no $\zeta$-dependent prefactor in \eqref{eq:renorm_moment_conv_nozeta}. This factor is naturally accounted for by the change in the lattice. Furthermore, we did not have to choose a lift from $\R/\Z$ to $\R$. Note also that the prelimit coordinates $\rank(p^{i-1}G_n) - c_n$ now live in $\Z+c_n$, an $n$-dependent shift of the lattice $\Z$ which nonetheless converges to the lattice $\Z+\zeta$. Furthermore, weak convergence of these random variables is clearly implied by the convergence in distribution of the related versions in \Cref{thm:group_convergence_intro}, which all live on the same lattice, so \Cref{thm:noround_groups_general} is immediate from \Cref{thm:group_convergence_intro}.

\begin{defi}\label{def:funny_cL}
For any $d \in \Z_{\geq 1}$, $p$ prime and $\zeta \in \R/\Z$, we define $\hcL_{d,p^{-1},\zeta}$ as the random variable valued in $\Sig_d + \zeta$ with moments 
\begin{equation}
  \E[p^{\la \cdot \hcL_{d,p^{-1},\zeta}}] = \frac{n_{max}(G_{\la'})}{|\la|!}
\end{equation}
for every $\la \in \Y_d$ (c.f. \Cref{def:cL_intro}).
\end{defi}

Note that the support depends on $\zeta$ but the moments do not.

\begin{theorem}\label{thm:noround_matrix_products}
Fix $p$ prime, let $\xi$ be a $\Z$-valued random variable such that $\xi \pmod{p}$ is nonconstant. For each $n \in \mathbb{Z}_{\geq 1}$ let $A_{i}^{(n)}, i \geq 1$ be iid $n \times n$ matrices over $\Z$ with iid $\xi$ entries. Let $(k(n))_{n \geq 1}$ be a sequence of natural numbers such that $k(n) \to \infty$ as $n \rightarrow \infty$ and $k(n) = O(e^{(\log n)^{1-\eps}})$ for some $0<\eps<1$. Define
\begin{equation}
G_n = \cok(A_{k(n)}^{(n)} \cdots A_{1}^{(n)})[p^\infty],
\end{equation}
let $(n_j)_{j \geq 1}$ be any subsequence such that $\lim_{j \to \infty} -\log _{p} k(n_{j}) =: \zeta$ exists in $\R/\Z$. Then for any $d \in \Z_{\geq 1}$,
\begin{equation}\label{eq:rmt_bulk_noround}
(\rank(p^{i-1}G_{n_j})-\log _{p}(k(n_{j})))_{1 \leq i \leq d} \rightarrow \hcL_{d,p^{-1},\zeta}
\end{equation}
weakly as $j \rightarrow \infty$. Furthermore, the exact same result holds for matrices over the $p$-adic integers $\Z_p$.
\end{theorem}

\Cref{thm:noround_matrix_products} follows from \Cref{thm:noround_groups_general} by our later moment computations, with essentially the same proof as \Cref{thm:matrix_product_intro} from \Cref{thm:group_convergence_intro}.

\section{Supporting lemmas for a single matrix}\label{sect:support} 

Throughout this section fix $a \in \BBZ_{> 1}$ and set\footnote{Although it suffices to focus only on $a=p^d$, in this section for generality we allow $a$ to be product of powers of many primes.} $R= \Z/a\Z$. Let $V=R^n$ with standard basis $\Bv_i,1 \leq i \leq n$. For $\sigma \subset [n]$ we denote by $V_{\sigma^c}$ the submodule generated by $\{\Bv_i: i \in \sigma^c\}$. Throughout the paper, to declutter notation we will write $(x_1,\ldots,x_n) \in R^n$ for usual (column) vectors, and similarly for vectors in e.g. $G^n$ where $G$ is an abelian group, rather than using the notation $(x_1,\ldots,x_n)^T$.

\begin{definition}\label{def:R_alpha_balanced}
Given real $\alpha \in (0,1/2]$, we say an $R$-valued random variable $\xi$ is \emph{$\al$-balanced} if for every prime $p|a$ we have
\begin{equation}\label{eqn:alpha_R}
\max_{r \in \Z/p\BBZ} \P(\xi\equiv r \pmod{p}) \le 1-\al.
\end{equation}
\end{definition}

Most of the results below are from \cite{W0,W1}. However, for later application, we will have to work out the constants as explicitly as possible.

\subsection{Codes} 
\begin{definition} Given $w \le n$, we say that $F \in \Hom(V,G)$ is a code of distance $w$ if for every $\sigma \subset [n]$ with $|\sigma| <w$ we have $F (V_{\sigma^c})=G$.
\end{definition}

Sometimes it is convenient to identify $F$ with the vector $(F(\Bv_1),\dots, F(\Bv_n)) \in G^n$, and we will usually abuse notation and view $F$ as a vector rather than a map. In particular, if $X = (x_1,\ldots,x_n) \in R^n$ is a vector, we write $\lang F, X \rang := \sum_{i=1}^n x_i F(\Bv_i)$; note this is not a usual dot product because $(F(\Bv_1),\dots, F(\Bv_n)) \in G^n$ and $(x_1,\ldots,x_n) \in R^n$ live in different spaces, though the formula is the same. If $M$ is an $n \times n$ matrix with entries in $R$, then for any $R$-module $G$, $M$ defines a linear map $G^n \to G^n$ by usual matrix multiplication, and we write $MF$ for the image of the vector $(F(\Bv_1),\dots, F(\Bv_n)) \in G^n$ under this map.

It is convenient to work with codes because the random walk $S_k = \sum_{i=1}^k x_i F(\Bv_i)$ (in discrete time indexed by $k=1,2,\ldots,n$) spreads out in $G$ very fast, as the following lemma shows. 

\begin{lemma}[{\cite[Lemma 2.1]{W1}}]\label{lemma:code:single:1} Assume that $x_i\in R$ are iid copies of $\xi$ satisfying \eqref{eqn:alpha_R}. Then for any code $F$ of distance $\delta n$ and any $g\in G$,
$$\left|\P(\lang F, X \rang  =g) - \frac{1}{|G|}\right| \le \exp(-\al \delta n/a^2),$$
where $X=(x_1,\dots,x_n)$.
\end{lemma}
In what follows, if not specified otherwise, $X$ is always understood as the random vector $(x_1,\dots, x_n)$ where $x_i$ are iid copies of $\xi$ satisfying \eqref{eqn:alpha_R} as in Lemma \ref{lemma:code:single:1}.

Using the above result, it is not hard to deduce the following matrix form.
\begin{lemma}\label{lemma:code:single:matrix:1} Assume that the entries of $M$ of size $n$ are iid copies of $\xi$ satisfying \eqref{eqn:alpha_R}. For any code $F$ of distance $\delta n$, for any vector $A \in G^n$ 
$$\left|\P(M F = A) - \frac{1}{|G|^{n}}\right| \le \frac{2n|G| \exp(- \al \delta n/a^2)}{|G|^{n}}.$$
\end{lemma}  
\begin{proof}
This is exactly the equation at the bottom of page 390 (two equations before Equation (5)) in \cite{W1}, with $u=0$ in the notation of that paper. 
\end{proof}

We will also need the following useful result, where for short by writing $x= y \pm z$ with $z\ge 0$ we mean 
\begin{equation}\label{eqn:pm}
x \in [y-z, y+z].
\end{equation}
\begin{lemma}\label{lemma:code:subgp} Let $\delta<1/8$. Assume that $F\in Hom(V,G)$ is a code of distance $\delta n$. Assume that the entries of the matrix $M$ of size $n$ are iid copies of $\xi$ satisfying \eqref{eqn:alpha_R}. Then for any $H \le G$
$$\P(M F \mbox{ is a code of distance $\delta n$ in $H_{}$}) =  |H_{}|^{n} \frac{1\pm 2n |G|\exp(-  \al \delta n/2a^2)}{|G|^n},$$
provided that $n \ge 8 (\log_2 |G|)^2$.
\end{lemma}

\begin{proof}[Proof of Lemma \ref{lemma:code:subgp}] First, by Lemma \ref{lemma:code:single:matrix:1}, for each $A$ a code of distance $\delta n$ of $H_{}$ we have 
$$|\P(M F=A) -  \frac{1}{|G|^{n}}| \le \frac{2n |G| \exp(- \al \delta n/a^2)}{|G|^{n}}.$$
It remains to count the number of codes of distance $\delta n$ in $H_{}$. 

\begin{claim}\label{claim:code:count} Let $\CC(H)$ be the number of codes of distance $\delta n$ in $H$. Then 
$$|\CC(H)| = (1\pm (1/2)^{n/2})|H|^n.$$
\end{claim}
\begin{proof} It suffices to show that a uniform element of $H^n$ has probability $\geq 1-(1/2)^{n/2}$ of being a code. Let $g_1,\dots, g_n$ be chosen independently uniformly from $H$. For each $I\subset [n]$ an index set of size $n -\lfloor \delta n \rfloor$, and for each $H'$ a proper subgroup of $H$, let $\CE_{I,H'}$ be the event that $g_i \in H'$ for all $ i\in I$. Then clearly $\P(\CE_{I,H'}) = (|H'|/|H|)^{|I|}$. Taking a union bound over the choices of $I \in \binom{[n]}{n-\lfloor \delta n\rfloor}$ and over $H' < H$ (noting that\footnote{This is a weaker version of Problem 4 of the 1996 Mikl\'os Schweitzer competition.} the number of subgroups of index $i$ of $H$ is at most $i^{1+(\log_2 |H|)^2}$) 
we obtain a bound
$$\Pr((g_1,\ldots,g_n) \text{ not code}) \leq \sum_{i | |H|, i\ge 2}  i^{1+(\log_2 |H|)^2} (1/i)^{n-\lfloor \delta n \rfloor} \times \binom{n}{\lfloor \delta n \rfloor} \le (1/2)^{n/2},$$
 as we assumed that $\delta<1/8$ and $n \ge 8 (\log_2 |G|)^2$ (in which case the first factor is bounded by 2, and the bound follows from $\binom{n}{k} \le (en/k)^k$). 
 \end{proof}

To complete the proof of Lemma \ref{lemma:code:subgp} we note that
 $$\left(\frac{1}{|G|^n}\pm  \frac{2 n |G| \exp(- \al \delta n/a^2)}{|G|^{n}}\right)  \times  (1+(1/2)^{n/2})|H|^n =  |H_{}|^{n} \frac{1 \pm 2n|G|\exp(-  \al \delta n/2a^2)}{|G|^n},$$
as $\delta \alpha / a^2 < 1/32$. 
\end{proof}

\subsection{Non-codes} Next, for non-code $F$, the random walk $\lang F, X \rang $ does not converge quickly to the uniform distribution on $G$. However it is likely to be uniform over the subgroup where the restriction of $F$ is a code. 

\begin{definition}
For $D =\prod_i p_i^{e_i}$ let 
$$\ell(D):= \sum_i e_i.$$ 
\end{definition}

In all results introduced below we remark that $F$ is not necessarily a surjection. 

\begin{definition}\label{def:depth} For a real $\delta>0$, the $\delta$-depth of $F \in \Hom(V,G)$ is the maximal positive integer $D$ such that there exists $\sigma \subset [n]$ with $|\sigma| < \ell(D) \delta n$ such that $D = |G/F(V_{\sigma^c})|$, or is 1 if there is no such $D$.
 \end{definition}
So roughly speaking the depth is large if there exists such $\sigma$ where $F(V_{\sigma^c})$ is a small subgroup of $G$. The reason for this definition of depth is the following lemma, which shows that depth encodes how much one has to restrict $F$ to obtain a code.

\begin{lemma}\label{lemma:restrict_to_code}
If $F \in \Hom(V,G)$ has $\delta$-depth $D>1$, and $\sigma \subset [n]$ is such that $D = |G/F(V_{\sigma^c})|$ and $|\sigma| < \ell(D) \delta n$, then the restriction $F|_{V_{\sigma^c}} \in \Hom(V_{\sigma^c}, F(V_{\sigma^c}))$ is a code of distance $\delta  |\sigma^c|$.
\end{lemma}

\begin{proof}
Suppose for the sake of contradiction that $F|_{V_{\sigma^c}}$ is not a code of distance $\delta |\sigma^c|$. Then there exists a set $\eta \subset \sigma^c$ with 
 $$|\eta| < \delta |\sigma^c| $$
such that $\Im(F|_{V_{(\eta \cup \sigma)^c}}) \subsetneq \Im(F_{V_{\sigma^c}})$. Hence $\tilde{D} := [G:\Im(F|_{V_{(\eta \cup \sigma)^c}})] > D$, and of course $D | \tilde{D}$. So
$$|\eta \cup \sigma| < \delta(n -|\sigma|) + |\sigma| = \delta n +(1-\delta)|\sigma| < \delta n + (1-\delta) \ell(D) \delta n <\delta (\ell(D)+1) n$$
and $\ell(\tilde{D}) \geq \ell(D)+1$, therefore 
$$|\eta \cup \sigma| < \ell(\tilde{D}) \delta n.$$
But this means that $\tilde{D}$ satisfies the condition in the definition of depth, and is larger than $D$, contradicting maximality, which completes the proof.
\end{proof}

\begin{lemma}\cite[Lemma 2.6]{W1}\cite[Lemma 3.1]{W0}\label{lemma:non-code:count:1} 
The number of $F\in \Hom(V,G)$ with depth $D$ is at most
$$K \binom{n}{\lceil \ell(D) \delta n\rceil -1} |G|^n D^{-n+\ell(D) \delta n},$$
where we can take $K = |G|^{\log_2 |G|}$. 
\end{lemma}
We also note that the above is similar to \cite[Lemma 2.6]{nguyen2022universality}, whose proof follows from \cite[Lemma 3.1]{W0} where $K$ can be taken to be $\max_{H \le G} \# \Hom(H, G^\ast)$.

\begin{lemma}\label{lemma:non-code:single:1} 
Let $F \in \Hom(V,G)$ have $\delta$-depth $D>1$ and $|G/F(V)|<D$. Then for any $g\in G$
$$\P(\lang F, X \rang   = g) \le (1-\al) \left(\frac{D}{|G|} + \exp\big(-\al \delta (1 - \ell(D) \delta)n/a^2\big)\right).$$

\end{lemma}
We remark that the assumption above is automatically true if $F$ is a surjection. This result is different from \cite[Lemma 2.7]{W1} in that $g$ is any element instead of just 0.

\begin{proof}[Proof of \Cref{lemma:non-code:single:1}] We follow the proof of \cite[Lemma 2.7]{W1}. Pick $\sigma \subset [n]$ with $|\sigma| < \ell(D) \delta n$ such that $D=|G/F(V_{\sigma^c})|$. Let $H=F(V_{\sigma^c})$. As $|G/F(V)|<D$, we have $\sigma \neq \emptyset$. 
We write 
\begin{align*}
\P(\lang F, X \rang  =g) &= \P\left(\sum_{i\in \sigma} x_i f_i +\sum_{i\in \sigma^c} x_i f_i =g\right) = \P\left(\sum_{i\in \sigma} x_i f_i  \in H_g \wedge \sum_{i\in \sigma^c} x_i f_i = g- \sum_{i\in \sigma} x_i f_i \right) \\
&= \P\left(\sum_{i\in \sigma} x_i f_i  \in H_g\right) \P\left(\sum_{i\in \sigma^c} x_i f_i = g- \sum_{i\in \sigma} x_i f_i | \sum_{i\in \sigma} x_i f_i  \in H_g\right),
\end{align*}
where $H_g$ is the coset of $H$ containing $g$. Now as $|G/F(V)|<D$, there exists $i_0 \in \sigma$ such that $f_{i_0} \notin H$. Since $x_{i_0}$ is $\al$-balanced, for any fixed values of $x_i, i \in \sigma \setminus i_0$ we have using the randomness of $x_{i_0}$ that
$$\P_{x_{i_0}}\left(\sum_{i\in \sigma }x_i f_i \in H_g\right)  \le 1-\al.$$

Furthermore, by Lemma \ref{lemma:restrict_to_code} $F(V_{\sigma^c})$ is a code of distance $\delta |\sigma^c| > (1-\ell(D)\delta)n$ over $H$, so  
\begin{align*}
  \left|\P\left(\sum_{i\in \sigma^c} x_i f_i = g- \sum_{i\in \sigma} x_i f_i | \sum_{i\in \sigma} x_i f_i  \in H_g\right) - \frac{1}{|H|}\right| & \le \exp\big(-\al \delta |\sigma^c|/a^2\big)\\ 
  &\le  \exp\big(-\al \delta (1 - \ell(D) \delta)n/a^2\big).
\end{align*}
Putting together we have 
$$\P\left(\sum_{i\in \sigma }x_i f_i \in H_g \wedge \sum_{i \in \sigma^c} x_i f_i =g-\sum_{i\in \sigma }x_i f_i\right) \le (1-\al) \left(\frac{1}{|H|} +\exp\big(-\al \delta (1 - \ell(D) \delta)n/a^2\big)\right),$$
completing the proof.
\end{proof}

Using this result, we can obtain similar bound in matrix form, the same way \cite[Lemma 2.8]{W1} was deduced from \cite[Lemma 2.7]{W1}.
\begin{lemma}\label{lemma:non-code:1:matrix'}
 If $F \in \Hom(V,G)$ has $\delta$-depth $D>1$ and $|G/F(V)|<D$ as in the previous lemma, then for any $A \in G^n$,
 $$\P(MF = A) \le  \exp(-\alpha n) \left(\frac{D}{|G|}\right)^n \exp\Big(n(|G|/D)\exp(-\alpha\delta (1 - \ell(D) \delta)n/a^2)\Big).$$
 
\end{lemma}

\begin{proof} By Lemma \ref{lemma:non-code:single:1},
\begin{align*}
\P(MF = A)  &= \P(\lang F, X_i \rang  = a_i, 1\le i\le n) \le  \left( (1-\al) \left(\frac{D}{|G|} + \exp\big(-\al \delta (1 - \ell(D) \delta)n/a^2\big)\right)
\right)^n.
\end{align*}

 This is bounded above by 
$$ \exp(-\alpha n) \left(\frac{D}{|G|}\right)^n \exp\Big(n(|G|/D)\exp(-\alpha\delta (1 - \ell(D) \delta)n/a^2)\Big),$$ completing the proof.
\end{proof}

\section{Counting surjections}\label{section:prod:moments}

This part mainly follows \cite{nguyen2022universality} but makes the constants explicit. We first introduce a definition that will be crucial to our work. 
\begin{definition}\label{def:n_k}
For a given $k\ge 0$ we let $n_k(G)$ denote the number of sequences of nested subgroups 
$$0=H_0\le H_1 \le H_2 \le \dots \le H_{k-1} \le H_k=G.$$ 
\end{definition}

Let $a,R,V$ be as in Section \ref{sect:support}. Throughout the section we write $\Hom(A,B)$ and  $\Sur(A,B)$ for the set of homomorphisms and surjective homomorphisms, respectively, from $A$ to $B$. All of our matrices $M$ in this section, if not specified otherwise, are in $\Mat_n(R)$ and their respective cokernels are 
$$\cok(M) = R^n/MR^n.$$

\subsection{Set-up}\label{S:mom}

We know from \cite{W1} that to understand the distribution of $\cok (M)$, it suffices to determine the moments of $\cok(M)$, i.e. the quantities $\E[\#\Sur(\cok (M), G)]$ for each finite abelian group $G$. 
To investigate each such moment, we recognize that each such surjection lifts to a surjection $V\ra G$ and so we have
\begin{equation}\label{E:expandF}
\E[\#\Sur(\cok (M), G)]=\sum_{F\in \Sur(V,G)}  \P(F(MV)=0 \mbox{ in $G$})= \sum_{F\in \Sur(V,G)}  \P(MF=0 \mbox{ in $G$}),
\end{equation}
where we view $F$ as a column vector $F=(F(\Bv_1),\dots, F(\Bv_n)) \in G^n$. By the independence of columns, we have
$$
\P(MF=0)= \prod_{j=1}^n \P(\lang F,X_j \rang=0),
$$
where $X_1,\dots, X_n$ are rows of $M$. So in the case of a single matrix, one must estimate these probabilities $\P(F(X_j)=0)$, which give the desired moments. In our situation we have random matrices $M_1, M_2,\dots, M_k$, and want to study $\P(M_1M_2 \cdots M_k  F=0)$ for surjections $F:V \to G$.

Recall $n_k(G)$ from Definition \ref{def:n_k}. Our key results in this section are generalizations of Lemma \ref{lemma:code:single:matrix:1} and Lemma \ref{lemma:non-code:1:matrix'}.

For the rest of this section we fix $\al, a$ and let 
$$c = \al /16 a^2.$$ 
We will assume that $n$ is sufficiently large (given $\al, a$) and that 
\begin{equation}\label{deltamagnitude}
\delta  = n^{-c}
\end{equation} 
and 
\begin{equation}\label{Gmagnitude}
|G| \le \exp(n^{c/8}).
\end{equation} 

Under these choices of $c,\delta$ and $G$, we restate Lemma \ref{lemma:non-code:1:matrix'} as follows.
\begin{lemma}\label{lemma:non-code:1:matrix} Assume that $\al, a,c, \delta$ and $G$ are as in Equations \eqref{deltamagnitude} and \eqref{Gmagnitude}. Suppose that $F \in \Hom(V,G)$ has $\delta$-depth $D>1$ and $|G/F(V)|<D$. Then for any $A \in G^n$,
 $$\P(MF = A) \le K_{0} \exp(-\al n)\frac{D^n}{|G|^n},$$
where
\begin{equation}\label{eq:def_k0}
K_0 := \max_{\substack{n \geq 1 \\ \delta=n^{-c} \\ G: |G| \leq \exp(n^{c/8}),1<D \le |G|}}  \exp\Big(n(|G|/D)\exp(-\alpha\delta (1 - \ell(D) \delta)n/a^2)\Big).
\end{equation}
\end{lemma}
Here we note that $\ell(D) \le \log_{2}|G|$, and hence $K_{0}$ is a finite constant depending on $c$.

To proceed, let us define two sequences $\eps_{i,n}, \eps_{i,n}'$ as follows. They will appear later as certain error bounds, but we will define them and prove they are small before discussing this (an interested reader may wish to skip to later in the section to get a sense of the role they will play).
\begin{defi}\label{def:eps_seq}
Set 
 $$\eps_{1,n} := 2 \exp(-\al \delta n/2a^2),  \eps_{1,n}' := K_0 \exp(-\al n),$$
 where $K_0$ is as in Equation \eqref{eq:def_k0}. In general for $i\ge 2$, set
$$\eps_{i,n} = 2n |G| \exp(-\al \delta n/2a^2) +  \eps_{i-1,n} + 2 n |G| \eps_{i-1,n}\exp(-\al \delta n/2a^2) + 2 \eps_{i-1,n}'  (n|G|)^{1+\delta n \log_2 |G|}$$
and
$$\eps_{i,n}'= \eps_{i-1,n}' (n|G|)^{2+\delta n \log_2 |G|} +K_0 \exp(-\al n)(1+ \eps_{i-1,n}).$$ 
\end{defi}

\begin{claim}\label{claim:eps-eps'} Let $n$ be sufficiently large given $\al, a$. Then for $k \le n^{c/8}$ we have 
$$\eps_{k,n}, \eps_{k,n}' \le \exp(-n^{c/4}).$$
\end{claim}

\begin{proof}[Proof of Claim \ref{claim:eps-eps'}]
For short let
\[
u:=2n|G| \exp(-\alpha\delta n/2a^2), \ \mbox{ and }\
A:=(n|G|)^{2+\delta n\log_2|G|}.
\]
Then the recursion above can be rewritten as 
$$\eps_{i,n} = u+  \eps_{i-1,n} + u \eps_{i-1,n} + (2A/n|G|) \eps_{i-1,n}'$$
and
$$\eps_{i,n}'= A\eps_{i-1,n}' +K_0 \exp(-\al n)(1+ \eps_{i-1,n}).$$ 
Note that as \(|G|\le e^{n^{c/8}}\), we have
\[
\log(n|G|)\le \log n+n^{c/8}\le 2 n^{c/8},
\]
and also
\[
\delta n\log_2|G| \le n^{1-c}n^{c/8}=n^{1-7c/8}.
\]
Note that we have suppressed a factor of $\log 2$ in the above, which is fine because we always assume $n$ is sufficiently large. Hence for all sufficiently large \(n\)
\[
\log A
=
\bigl(2+\delta n\log_2|G|\bigr)\log(n|G|)
\le 
2 (2+n^{1-7c/8})n^{c/8}
\le 4 n^{1-3c/4}.
\]
Therefore 
\begin{equation}\label{eq:A-growth-proof}
A\le \exp(4n^{1-3c/4}).
\end{equation}

Also, as $u=2n|G|\exp(-\alpha n^{1-c}/2a^2)$, for some large constant $C$ we can write
\begin{equation}\label{eq:u-small-proof}
u\le \exp(-n^{1-c}/C).
\end{equation}
Consequently, for any $1\le i \le k$
$$(u+1)^{i} < 2$$
for all $n$ sufficiently large. Our main claim is the following.
\begin{fact}\label{cl:1} For every \(1\le i\le k\),
\begin{equation}\label{eq:induction-epsprime}
\eps'_{i,n}\le 2K_0\, i\, A^{\,i-1}e^{-\alpha n},
\end{equation}
and
\begin{equation}\label{eq:induction-eps}
\eps_{i,n}\le i(1+u)^i u + 4K_0\, i^2 (1+u)^i A^{\,i-1}e^{-\alpha n}.
\end{equation}
\end{fact}

\begin{proof}(of Fact \ref{cl:1})
We prove these simultaneously by induction on \(i\).

For \(i=1\), we have
\[
\eps'_{1,n}=K_0e^{-\alpha n}\le 2K_0e^{-\alpha n},
\]
so \eqref{eq:induction-epsprime} holds. Also
\[
\eps_{1,n}\le u \le (1+u)u + 4K_0(1+u)e^{-\alpha n},
\]
so \eqref{eq:induction-eps} also holds for \(i=1\).

Assume now that \eqref{eq:induction-epsprime} and \eqref{eq:induction-eps} hold for \(i-1<n^{c/8}\). We first estimate \(\eps'_{i,n}\). Since \(\eps_{i-1,n}\ge 0\), from the recursion,
\[
\eps'_{i,n}
=
A\eps'_{i-1,n}+K_0e^{-\alpha n}(1+\eps_{i-1,n}).
\]
At this point we use the induction bound \eqref{eq:induction-eps}. Since \(ku\to 0\) and \(k^{2}A^{k-1}e^{-\alpha n}\to 0\) uniformly for \(k\le n^{c/8}\), by the induction hypothesis for $\eps_{i-1,n}$, for all sufficiently large \(n\) we have
\begin{align*}
\eps_{i-1,n}&\le  (i-1)(1+u)^{i-1} u + 4K_0\, (i-1)^2 (1+u)^{i-1} A^{\,i-2}e^{-\alpha n} \le 1
\end{align*}
Hence
\[
\eps'_{i,n}\le A\eps'_{i-1,n}+2K_0e^{-\alpha n}.
\]
Using the induction hypothesis \eqref{eq:induction-epsprime},
\[
\eps'_{i,n}
\le
A(2K_0(i-1)A^{i-2}e^{-\alpha n}) + 2K_0e^{-\alpha n}
=
2K_0e^{-\alpha n}\bigl((i-1)A^{i-1}+1\bigr).
\]
Since \(1\le A\le A^{i-1}\), and therefore
\(
(i-1)A^{i-1}+1\le iA^{i-1}.
\)
Thus
\[
\eps'_{i,n}\le 2K_0\, i\, A^{i-1}e^{-\alpha n},
\]
which proves \eqref{eq:induction-epsprime} for \(i\).

Next, from the recursion for \(\eps_{i,n}\),
\[
\eps_{i,n}
=
u+\eps_{i-1,n}+u\eps_{i-1,n}+(2A/n|G|)\eps'_{i-1,n}
\le
(1+u)\eps_{i-1,n}+u+2A\eps'_{i-1,n}.
\]
Applying the induction hypotheses \eqref{eq:induction-eps} and \eqref{eq:induction-epsprime},
\begin{align*}
\eps_{i,n} &\le
(1+u)\Bigl((i-1)(1+u)^{i-1}u
+4K_0(i-1)^2(1+u)^{i-1}A^{i-2}e^{-\alpha n}\Bigr)
+u
+4K_0(i-1)A^{i-1}e^{-\alpha n}\\
&\le
(i-1)(1+u)^i u+u
+4K_0(i-1)^2(1+u)^iA^{i-2}e^{-\alpha n}+4K_0(i-1)A^{i-1}e^{-\alpha n}\\
&\le  i(1+u)^i u + 4K_0 i^2(1+u)^iA^{i-1}e^{-\alpha n},
\end{align*}
confirming \eqref{eq:induction-eps}, where we used the fact that
\[
(i-1)(1+u)^i u+u\le i(1+u)^i u
\]
and
\[
(i-1)^2(1+u)^iA^{i-2}+(i-1)A^{i-1}
\le
\bigl((i-1)^2+(i-1)\bigr)(1+u)^iA^{i-1}
\le i^2(1+u)^iA^{i-1}.
\]
 This closes the induction.
\end{proof}

We now set \(i=k\), where \(k\le n^{c/8}\). By \eqref{eq:A-growth-proof},
\(
(k-1)\log A\le n^{c/8}4n^{1-3c/4}=4n^{1-5c/8},
\)
and so
\[
A^{k-1}\le \exp(4n^{1-5c/8}).
\]
Hence
\[
\eps'_{k,n} \le 2K_{0} kA^{k-1}e^{-\alpha n} \le 2K_{0} n^{c/8} \exp(4n^{1-5c/8}) e^{-\alpha n} \le \exp(-\alpha n/2) \le  \exp(-n^{c/4})
\]
for all sufficiently large \(n\). 

Similarly, by \eqref{eq:u-small-proof},
\[
k(1+u)^k u\le 2k u\le \exp(-n^{1-c}/2C)
\]
and
\[
k^2(1+u)^kA^{k-1}e^{-\alpha n}
\le 2k^{2}   \exp(4n^{1-5c/8}) e^{-\alpha n}
\le \exp(-\alpha n/2).
\]
It thus follows that, by using \eqref{eq:induction-eps} 
$$\eps_{k,n}\le k(1+u)^k u + 4K_0\, k^2 (1+u)^k A^{\,k-1}e^{-\alpha n} \le \exp(-n^{c/4}).$$
\end{proof}

\begin{proposition}\label{prop:single:k} The following holds for $n$ sufficiently large: 

\begin{enumerate}[label=(\roman*)]
\item  (Code) assume that $F$ spans $H_k=G$ and is a code of distance $\delta n$ in $G$. Then 
$$\Big |\P( M_1\cdots M_k  F=0) -\frac{n_k(G)}{|G|^n} \Big |\le  \eps_{k,n} \frac{n_k(G)}{|G|^n}.$$ 
\item (Non-code) Assume that $F$ spans $H_k=G$ and the $\delta$-depth of $F$ is $D_k\ge 2$. Then 
$$\P(M_1\cdots M_k  F=0) \le \eps_{k,n}'  n_k(G) \frac{D_k^n}{|G|^n}.$$
\end{enumerate}
\end{proposition}

\begin{proof} We prove (i) and (ii) together by induction on $k$, assuming both (i) and (ii) hold for $k-1$ as the inductive hypothesis. When $k=1$, (i) and (ii) follow from Lemma \ref{lemma:code:single:matrix:1} and Lemma \ref{lemma:non-code:1:matrix} respectively with
$$\eps_{1,n} = 2 \exp(-\al \delta n/2a^2),  \eps_{1,n}' = K_0 \exp(-\al n).$$
 Next we consider $k\ge 2$.

{\bf Codes.} We first prove (i) by working with $F$ a code of distance $\delta n$. 

Let $H_{k-1}$ be a subgroup of $H_k=G$. We consider the event (in the $\sigma$-algebra generated by $M_k$) that $M_kF$ spans $H_{k-1}$ in two ways 
\begin{enumerate}[label=(\arabic*)]
\item $M_k F$ is a code of distance $\delta n$ in $H_{k-1}$; 
\item $M_k F$ is not a code of distance $\delta n$, and hence has $\delta$-depth $D_{k-1}\ge 2$ in $H_{k-1}$.
\end{enumerate}
For the first case, we apply the induction hypothesis for (i) to obtain
\begin{multline}
  \Big|\P_{M_1,\dots, M_{k-1}}(M_1\dots M_{k-1}(M_k F)=0|  \text{$M_k F$ is $\delta n$ code in $H_{k-1}$}) -\frac{n_{k-1}(H_{k-1})}{|H_{k-1}|^n}\Big| \\ \le \eps_{k-1,n}  \frac{n_{k-1}(H_{k-1}) }{|H_{k-1}|^n}.
\end{multline}
For the second case, we also apply the induction hypothesis for (ii) to obtain
\begin{multline}
  \P_{M_1,\dots, M_{k-1}}\left(M_1\dots M_{k-1}(M_k F)=0|  \text{ $M_k F$ has $\delta$-depth $D_{k-1}\ge 2$ in $H_{k-1}$}\right) \\ \le \eps_{k-1,n}' n_{k-1}(H_{k-1})\frac{D_{k-1}^n}{|H_{k-1}|^n}.
\end{multline}
Hence
\begin{align*}
&\ \P_{}\left(\prod_{i=1}^k M_i F=0, \text{$M_k F$ spans $H_{k-1}$}\right) \\ 
&=\P_{M_1,\dots, M_{k-1}}(M_1\dots M_{k-1}(M_k F)=0|  \text{$M_k F$ is $\delta n$-code in $H_{k-1}$})  \\
& \times \P(\text{$M_k F$ is $\delta n$-code in $H_{k-1}$})\\ 
& +\sum_{\substack{D_{k-1} \geq 2 \\ D_{k-1} \big{\vert} |H_{k-1}|}} \P_{M_1,\dots, M_{k-1}}(M_1\dots M_{k-1} (M_k F)=0| \text{ $M_k F$ has $\delta$-depth $D_{k-1}$ in $H_{k-1}$}) \\
& \times \P(\text{$M_k F$ has $\delta$-depth $D_{k-1}$ in $H_{k-1}$}) \\
 &=: S_1(H_{k-1}) + \sum_{\substack{D_{k-1} \geq 2 \\ D_{k-1} \big{\vert} |H_{k-1}|}}S_2(H_{k-1}, D_{k-1}).
\end{align*}

For the first sum, by Claim \ref{claim:code:count}, and then by Lemma \ref{lemma:code:subgp} and the inductive hypothesis for (i) we have (recalling the $\pm$ notation from \eqref{eqn:pm})
\begin{align*}
S_1(H_{k-1})&= \left(\frac{n_{k-1}(H_{k-1})}{|H_{k-1}|^n} \pm \eps_{k-1,n} \left(\frac{n_{k-1}(H_{k-1})}{|H_{k-1}|^n}\right)\right)   |H_{k-1}|^n \frac{1\pm 2n |G| \exp(-\al \delta n/2a^2)}{|G|^n}\\
&=  \frac{n_{k-1}(H_{k-1})}{|G|^n}\Big(1 \pm  2 n |G| \exp(-\al \delta n/2a^2) \pm  \eps_{k-1,n} \pm 2n |G| \eps_{k-1,n}\exp(-\al \delta n/4a^2) \Big).
\end{align*}
For the second sum, for each $D_{k-1}$ we apply Lemma \ref{lemma:non-code:count:1} and Lemma \ref{lemma:code:single:matrix:1} to bound
\begin{multline}\label{eq:some_depth_prob_bound}
\P(\text{$M_k F$ has $\delta$-depth $D_{k-1}$ in $H_{k-1}$}) \\  \leq |G|^{\log_2 |G|} \binom{n}{\lceil \ell(D_{k-1}) \delta n\rceil -1} |H_{k-1}|^n D_{k-1}^{-n+\ell(D_{k-1}) \delta n}  \frac{1+ 2n |G| \exp(-\al \delta n/2 a^2)}{|G|^n}
\end{multline}
and apply the inductive hypothesis for (ii) to bound 
\begin{multline}\label{eq:ind(ii)}
\P_{M_1,\dots, M_{k-1}}(M_1\dots M_{k-1} (M_k F)=0| \text{ $M_k F$ has $\delta$-depth $D_{k-1}$ in $H_{k-1}$}) \\ \le \eps_{k-1,n}'  n_{k-1}(H_{k-1}) \frac{D_{k-1}^n}{|H_{k-1}|^n}.
\end{multline}
Combining \eqref{eq:some_depth_prob_bound} with \eqref{eq:ind(ii)} yields
\begin{align*}
S_2(H_{k-1},D_{k-1})& \le  \eps_{k-1,n}' n_{k-1}(H_{k-1})  \frac{D_{k-1}^n}{|H_{k-1}|^n} \times \\
&\times |G|^{\log_2 |G|} \binom{n}{\lceil \ell(D_{k-1}) \delta n\rceil -1} |H_{k-1}|^n D_{k-1}^{-n+\ell(D_{k-1}) \delta n}  \frac{1+ 2n |G| \exp(-\al \delta n/2 a^2)}{|G|^n}\\
& \le 2 \eps_{k-1,n}'  n_{k-1}(H_{k-1})   |G|^{\log_2 |G|} \binom{n}{\lceil \ell(D_{k-1}) \delta n\rceil -1} D_{k-1}^{\ell(D_{k-1}) \delta n}   \frac{1}{|G|^n}\\
& \le 2 \eps_{k-1,n}' (n|G|)^{\delta n \log_2 |G| } n_{k-1}(H_{k-1})   \frac{1}{|G|^n},
\end{align*}
where we used the fact that $\ell(D_{k-1}) \le \log_2 D_{k-1} \le \log_2 |G|$ and $ \binom{n}{\lceil \ell(D_{k-1}) \delta n\rceil -1} \le n^{\lceil \ell(D_{k-1}) \delta n\rceil -1}$, and $n$ is sufficiently large, with $\delta$ and $G$ from \eqref{deltamagnitude} and \eqref{Gmagnitude} respectively.


Summing over divisors $D_{k-1}$ of $|H_{k-1}|$ (for which we generously bound it from above by $|G|$),
\begin{align*}
\sum_{D_{k-1} \geq 2, D_{k-1}\big{\vert} \left\vert H_{k-1}\right\vert} S_2(H_{k-1},D_{k-1})  &\le |G| 2 \eps_{k-1,n}'  |G|^{\delta n\log_2 |G| }   n_{k-1}(H_{k-1})  \frac{1}{|G|^n} \\
&\le 2 \eps_{k-1,n}'  (n|G|)^{1+\delta n \log_2 |G|}  n_{k-1}(H_{k-1})   \frac{1}{|G|^n}.
\end{align*} 
Summing over $H_{k-1}\le H_k$ we thus obtain
\begin{align}\label{eqn:code:M_k:1}
\begin{split}
\P_{}\left(\prod_{i=1}^k M_i F=0\right) &=\sum_{H_{k-1}} \P_{}\left(\prod_{i=1}^k M_i F=0 \wedge \text{$M_k F$ spans $H_{k-1}$}\right)   \\
& = \sum_{H_{k-1}} S_1(H_{k-1}) + \sum_{\substack{D_{k-1} \geq 2 \\ D_{k-1} \big{\vert} |H_{k-1}|}}S_2(H_{k-1}, D_{k-1})  \\
 & = \frac{n_k(G)}{|G|^n}\Big(1 \pm  2n |G| \exp(-\al \delta n/4a^2) \pm  \eps_{k-1,n} \pm 2 n |G| \eps_{k-1,n}\exp(-\al \delta n/4a^2) \Big)\\
&\pm 2 \eps_{k-1,n}'  (n|G|)^{1+\delta n \log_2 |G|}      \frac{ n_{k}(G)}{|G|^n}\\
&= \frac{n_k(G)}{|G|^n}\Big(1 \pm  2n |G| \exp(-\al \delta n/4a^2) \pm  \eps_{k-1,n} \\
&
\pm 2n|G| \eps_{k-1,n}\exp(-\al \delta n/4a^2) \pm  2 \eps_{k-1,n}'  (n|G|)^{1+\delta n \log_2 |G|}  \Big).
\end{split}
\end{align}
Substituting in the recurrence for $\eps_{k,n}$ in \Cref{def:eps_seq} completes the inductive step for part (i).

{\bf Non-codes.} We next prove (ii) by working with $F$ of $\delta$-depth $D_k\ge 2$, where $D_k$ also divides $|H_k|=|G|$. Let $H_{k-1}$ be a subgroup of $H_k$. Similarly to the previous part, we again compute the probability that $M_kF$ spans $H_{k-1}$ in the two possible ways:

\begin{enumerate}[label=(\arabic*)]
\item $M_k F$ is a code of distance $\delta n$ in $H_{k-1}$; 
\item $M_k F$ is not a code of distance $\delta n$, and hence has $\delta$-depth $D_{k-1}\ge 2$ in $H_{k-1}$.
\end{enumerate}

For the first case, the probability with respect to $M_k$ is bounded by
$$\P_{M_k}(   \text{$M_k F$ is a code of distance $\delta n$ in $H_{k-1}$} ) \le K_0 |H_{k-1}|^n   \exp(-\al n)  \frac{D_{k}^n}{|G|^n}$$
by bounding the number of codes by $|H_{k-1}|^n$ and applying Lemma  \ref{lemma:non-code:1:matrix}. Hence, by induction and by the independence of $M_1,\dots, M_k$
\begin{align*}
& \P_{}(M_1\cdots M_k  F=0 \text{ and $M_k F$ is a code of distance $\delta n$ in $H_{k-1}$}) \\
& \le \left(\frac{n_{k-1}(H_{k-1})}{|H_{k-1}|^n}+\eps_{k-1,n}  \frac{n_{k-1}(H_{k-1})}{|H_{k-1}|^n}\right) \times  K_0 |H_{k-1}|^n   \exp(-\al n)  \frac{D_{k}^n}{|G|^n} \\
 &\le K_0 \exp(-\al n)(1+ \eps_{k-1,n})\frac{n_{k-1}(H_{k-1}) D_{k}^n}{|G|^n}.
\end{align*}
Summing over the subgroups $H_{k-1}$, we obtain
\begin{multline}\label{eqn:non-code:M_k:1}
\P_{}(M_1\cdots M_k  F=0 \text{ and $M_k F$ is a $\delta n$ code in $H_{k-1}$ for some $H_{k-1} \leq H_k$}) \\ \le  K_0 \exp(-\al n)(1+ \eps_{k-1,n})\frac{n_{k}(G) D_{k}^n}{|G|^n}. 
\end{multline}

For the second case (2), the probability with respect to $M_k$, by Lemma \ref{lemma:non-code:count:1} and Lemma  \ref{lemma:non-code:1:matrix}, is bounded by
\begin{multline}
  \P_{M_k}(   \text{$M_k F$ is of $D_{k-1}$-depth in $H_{k-1}$} ) \\ \le  |G|^{\log_2 |G|}  \binom{n}{\lceil \ell(D_{k-1}) \delta n\rceil -1} |H_{k-1}|^n D_{k-1}^{-n+\ell(D_{k-1}) \delta n}  \times  K_0\exp(-\al n)  \frac{D_{k}^n}{|G|^n} .
\end{multline}
Hence, by induction (applied to $M_1 \cdots M_{k-1}$ with the starting vector $(M_kF)$)
\begin{align*}
& \P_{}(M_1\cdots M_k  F=0 \text{ and $M_k F$ is of $\delta$-depth $D_{k-1}$ in $H_{k-1}$}) \\
&= \P_{}(M_1\cdots M_k  F=0 |\text{$M_k F$ is of $\delta$-depth $D_{k-1}$ in $H_{k-1}$}) \cdot \P(\text{$M_k F$ is of $\delta$-depth $D_{k-1}$ in $H_{k-1}$})\\
& \le \eps_{k-1,n}'  n_{k-1}(H_{k-1})  \frac{D_{k-1}^n}{|H_{k-1}|^n} \times |G|^{\log_2 |G|}  \\ 
&\times \binom{n}{\lceil \ell(D_{k-1}) \delta n\rceil -1} |H_{k-1}|^n D_{k-1}^{-n+\ell(D_{k-1}) \delta n}  \times  K_0\exp(-\al n)  \frac{D_{k}^n}{|G|^n} \\
&\le \eps_{k-1,n}'   (n|G|)^{1+\delta n \log_2 |G|}   n_{k-1}(H_{k-1}) \frac{D_{k}^n}{|G|^n},
\end{align*}
provided that $n$ is sufficiently large.

Summing over $D_{k-1}$ a divisor of $|H_{k-1}|$ (for which, very generously, there are at most $|G|$ of such), and then over the subgroup $H_{k-1}$
\begin{align}\label{eqn:non-code:M_k:2}
\begin{split}
& \P_{}(M_1\cdots M_k  F=0 \text{ and $M_k F$ is of $D_{k-1}$-depth in $H_{k-1}$ for some $D_{k-1}\ge 2$ and subgroup $H_{k-1}$}) \\ 
& \le \eps_{k-1,n}' (n|G|)^{2+\delta n \log_2 |G|}   n_{k}(G) \frac{D_{k}^n}{|G|^n},
\end{split}
\end{align}
proving our upper bound for non-codes $F$. Combining \eqref{eqn:non-code:M_k:1} with \eqref{eqn:non-code:M_k:2} and the recurrence for $\eps'_{k,n}$ in \Cref{def:eps_seq} completes the inductive step for part (ii). 
\end{proof}

Using the proof of Proposition \ref{prop:single:k} and Claim \ref{claim:eps-eps'} above we obtain

\begin{theorem}[Asymptotic moments of matrix products]\label{theorem:surmoment:k} Let $a \geq 2$ and $R=\Z/a\Z$, $G$ be any finite abelian group whose exponent is divisible by $a$, and $M_1,\ldots,M_k$ be random matrices in $\Mat_n(R)$ with iid $\alpha$-balanced entries. Let $c= \al/16 a^2$ and assume that 
$$|G| \le \exp(n^{c/8})$$
and
$$k \le n^{c/8}.$$
Then
$$\Big |\E \left[\# \Sur(\cok(M_1\cdots M_k ),G)) - n_k(G)\right]\Big| \le e^{-n^{c/4}} n_k(G),$$
for sufficiently large $n$.
\end{theorem}

\begin{proof}[Proof of Theorem \ref{theorem:surmoment:k}] By \eqref{E:expandF} and using Claim \ref{claim:eps-eps'} to bound the right hand side below, it suffices to show that 
\begin{multline}\label{eqn:prod:sum}
\Big|\sum_{F\in \Sur(V,G)}  \P(M_1\cdots M_k  F=0 \mbox{ in $G$}) - n_k(G)\Big| \\ \le n_k(G)\Big(2\eps_{k,n}+K_0 \exp(-\al n)(n|G|)^{\delta n \log_2 |G|}+\eps_{k-1,n}' (n|G|)^{2+2\delta n \log_2 |G|} \Big).
\end{multline}
From Proposition \ref{prop:single:k} (i) and Claim \ref{claim:code:count}, we sum over $F$ as codes of distance $\delta n$ in $G$ to obtain

\begin{align}\label{eqn:code:sum:1}
\begin{split}
\sum_{F \text{ code of distance $\delta n$ in $G$}}\P_{}\left(\prod_{i=1}^k M_i F=0\right)& = |G|^n(1 \pm (1/2)^{n/2}) \times \left(\frac{n_k(G)}{|G|^n}(1\pm \eps_{k,n})\right)  \\
& = n_k(G)(1\pm 2\eps_{k,n}),
\end{split}
\end{align}
where we used the fact that, as $\delta=n^{-c}$, $\eps_{k,n} \ge \eps_{1,n} = 2 \exp(-\al \delta n/2a^2) \ge (1/2)^{n/2-1}$. 

From \eqref{eqn:non-code:M_k:1} and Lemma \ref{lemma:non-code:count:1}, for each $D_k$ as divisor of $|G|$, summing over non-codes $F$ of $\delta$-depth $D_k$
\begin{align}\label{eqn:noncode:code:1}
\begin{split}
& \P_{}(\mbox{ $\exists F$ of depth $D_k$ in $G$ such that } M_1\cdots M_k  F =0 \text{ and $M_k F$ is a $\delta n$ code in $H_{k-1}$ for some $H_{k-1}$}) \\
&\le  |G|^{\log_2 |G|} \binom{n}{\lceil \ell(D_{k}) \delta n\rceil -1} |G|^n D_{k}^{-n+\ell(D_{k}) \delta n}  K_0 \exp(-\al n) \frac{n_{k}(G) D_{k}^n}{|G|^n}\\
& \le K_0 \exp(-\al n)(n|G|)^{\delta n \log_2 |G|}   n_{k}(G),
\end{split}
\end{align}
provided that $n$ is sufficiently large. 

Also, from \eqref{eqn:non-code:M_k:2} and Lemma \ref{lemma:non-code:count:1}, 
\begin{align}\label{eqn:noncode:noncode}
\begin{split}
& \P(\text{$\exists F$ of depth $D_k$ in $G$ with } M_1\cdots M_k  F=0 \text{ and $M_k F$ is of depth $D_{k-1}$ in $H_{k-1}$} \\
&\text{ for some $D_{k-1}\ge 2, H_{k-1} \le H_k$}) \\ 
&\le |G|^{\log_2 |G|} \binom{n}{\lceil \ell(D_{k}) \delta n\rceil -1} |G|^n D_{k}^{-n+\ell(D_{k}) \delta n}  \eps_{k-1,n}' (n|G|)^{2+\delta n \log_2 |G|}   n_{k}(G) \frac{D_{k}^n}{|G|^n}  \\
&\le  \eps_{k-1,n}' (n|G|)^{2+2\delta n \log_2 |G|}  n_{k}(G),
\end{split}
\end{align}
 as  $n$ is sufficiently large. 

Summing \eqref{eqn:noncode:code:1}, \eqref{eqn:noncode:noncode} over all divisor $D_k$ of  $|G|$, together with \eqref{eqn:code:sum:1} we obtain \eqref{eqn:prod:sum} as claimed.
\end{proof}

We then deduce the following, showing how moment bounds for finite rings can be applied for infinite ones such as $\Z$ and $\Z_p$.

\begin{corollary}\label{cor:sur:uniform:asym} Let $M_1,\dots, M_k$ have iid entries in $\Z$ or $\Z_p$ which are $\al$-balanced modulo $p$. Then for any $p$-group $G$ with exponent $p^L$ and
$$|G| \le \exp(n^{\al/(128 p^{2L})})$$
and
$$k \le n^{\al/(128 p^{2L})}$$
we have 
$$\Big |\E (\# \Sur(\cok(M_1\cdots M_k),G)) - n_k(G)\Big| \le e^{-n^{\al/(64p^{2L})}} n_k(G) .$$
\end{corollary}
\begin{proof}
Let $p^L$ be the exponent of $G$. First note that for any abelian $p$-group $H$, 
\begin{equation}\label{eq:reduce_surj_simple}
\#\Sur(H,G) = \#\Sur(H/p^LH,G),
\end{equation}
as any surjection from $H$ to $G$ automatically annihilates $p^{L}H$. Note also, with the notation $\widetilde{M} := M \pmod{p^L} \in \Mat_n(\Z/p^L\Z)$ for $M \in \Mat_n(\Z_p)$, that $\cok(\widetilde{M}) = \cok(M)/p^L \cok(M)$. Combining with \eqref{eq:reduce_surj_simple} yields that  
\begin{equation}
\# \Sur(\cok(M_1\cdots M_k \Z_p^n),G) = \# \Sur(\cok(\widetilde{M}_1\cdots \widetilde{M}_k (\Z/p^{L}\Z)^n),G).
\end{equation}
The result now follows from Theorem \ref{theorem:surmoment:k} applied with $R=\Z/p^L\Z$ and $c= \al/(16 p^{2L})$ and matrices $\widetilde{M}_1,\ldots,\widetilde{M}_k$, which are $\alpha$-balanced.
\end{proof}

\section{Proof of \Cref{thm:matrix_product_intro}}\label{sect:n_k}

The only missing basic ingredient is some simple bounds and asymptotics on the function $n_k(G)$ defined in \Cref{def:n_k}. We will prove these, then prove \Cref{thm:matrix_product_intro}, and finally prove that \Cref{def:cL_intro} indeed defined a valid random variable.

Let $|G_{\mu,\nu}|$ be the number of subgroups isomorphic to $G_\mu$ of $G_\nu$.
We need the following estimate, which can be deduced from  \cite{Butler}; see also \cite[Lemma 7.4]{W0}.

\begin{lemma}\label{lemma:sbgp_num_bound}
For $\mu\le \nu$, we have
\[
p^{\sum_{i=1}^{\nu_1}\mu_i'(\nu_i'-\mu_i')}
\ \le\ 
|G_{\mu,\nu}|
\ \le\ 
p^{\sum_{i=1}^{\nu_1}\mu_i'(\nu_i'-\mu_i')}\,(p^{-1};p^{-1})_\infty^{-\nu_1}.
\]
In particular,
\[
p^{\sum_{i=1}^{\nu_1}\mu_i'(\nu_i'-\mu_i')}
\ \le\ 
|G_{\mu,\nu}|
\ \le\ 
p^{\sum_{i=1}^{\nu_1}\mu_i'(\nu_i'-\mu_i')}\,(2^{-1};2^{-1})_\infty^{-\nu_1}.
\]
\end{lemma}



The following result explains the appearance of $n_{\max}$ in the moments of our limiting random variable.

\begin{lemma}\label{lemma:n_k:largek} Recalling the notations $n_k$ and $n_{\max}$ from \Cref{def:n_k} and \Cref{def:cL_intro} respectively, we have
\begin{equation}\label{eq:limit_sbgp_seq}
\lim_{k \to \infty} \frac{n_k(G_{\la})}{\binom{k}{|\la|}} = n_{\max}(G_\la).
\end{equation}
Furthermore, the bounds
\begin{equation}
\label{eq:nmax_bounds}
p^{ \sum_{i=1}^{\la_1} \binom{\la_i'}{2}} \le n_{\max}(G_\la) \le d^{|\la|} (\prod_{k=1}^\infty \frac{1}{1-p^{-k}})^{d |\la|} p^{ \sum_{i=1}^{\la_1} \binom{\la_i'}{2}}
\end{equation}
hold for all $\la \in \Y$ with $\la' \in \Y_d$. 
\end{lemma}

\begin{proof} For each tuple $0< i_1<\dots <i_{m-1} < |\la|$, let $\CM_{i_1,\dots, i_{m-1}}$ be the collection of all sequences of nested $p$-groups $\{0\}=G_0 < G_1 < G_2 <\dots < G_{m-1} <  G_\la$, where $G_j$ has size $p^{i_j}$. For a given sequence of $\{G_i\}$ in  $\CM_{i_1,\dots, i_{m-1}}$,  the number of nested sequences (in the definition of $n_k(G)$) of form $\{0\}=H_0 \le H_1 \le H_2 \le \dots \le H_k=G_\la$ whose ``skeleton" is exactly the sequence $\{G_i\}$ (that is, as  sets $\{H_0,\dots, H_k\} = \{G_0,\dots, G_{m-1}, G_\la \}$) is $\binom{k}{m}$. Thus the total number of nested group sequences $H_i$ whose skeleton are in $\CM_{i_1,\dots, i_{m-1}}$ is exactly $\binom{k}{m}|\CM_{i_1,\dots, i_{m-1}}|$. 

Note that if $k$ is sufficiently large given $|\la|$ then
$$\sum_{1\le m\le |\la|}\sum_{0< i_1<\dots <i_{m-1} < |\la|} \binom{k}{m}|\CM_{i_1,\dots, i_{m-1}}| = (1+O_{|\la|}(1/k)) \binom{k}{|\la|} |\CM_{1,2,\dots, |\la|-1}|,$$ 
where the implied constant might depend on $|\la|$ (for instance we can take it very crudely to be $|\la|^{|\la|}$). But $|\CM_{1,2,\dots, |\la|-1}| = n_{\max}(G_\la)$, which proves \eqref{eq:limit_sbgp_seq}.

For the bounds \eqref{eq:nmax_bounds} we must enumerate $\CM_{1,2,\dots, |\la|-1}$, so we need to count the number of nested partitions (corresponding to the nested $p$-subgroups) of types ${\la^{(1)}}<\dots < {\la^{(i)}} < {\la^{(i+1)}}< \dots < {\la^{(|\la|)}}=\la$, where $|\la^{(i+1)}| = 1+ |\la^{(i)}|$. These are just standard Young tableaux, for which exact formulas exist, but all we need here is a crude bound. 
 
Recall from Lemma \ref{lemma:sbgp_num_bound} that for any $\mu \leq \nu$, 
$$p^{\sum_{i=1}^{\nu_1} \mu_{i}' (\nu_i' - \mu_i')} \le |G_{\mu,\nu}|  \le  \prod_{i=1}^{\nu_1} (\prod_{k=1}^\infty \frac{1}{1-p^{-k}}) p^{\sum_{i=1}^{\nu_1} \mu_{i}' (\nu_i' - \mu_i')},$$
 where the latter can be bounded from above by $(\prod_{k=1}^\infty \frac{1}{1-p^{-k}})^d p^{\sum_{i=1}^{\nu_1} \mu_{i}' (\nu_i' - \mu_i')} $ if $\nu_1 \le d$.

Now if we look at the (multiplicative) contribution (over $1\le j\le |\la|-1$) of $p^{\sum_{i=1}^{\nu_1} \mu_{i}' (\nu_i' - \mu_i')}$ over each pair $(\mu,\nu) = (\la^{(j)}, \la^{(j+1)})$ from (any)  sequence ${\la^{(1)}}<\dots < {\la^{(i)}} < {\la^{(i+1)}}< \dots < {\la^{(|\la|)}}=\la$, we see that 
$$\prod_{j=1}^{|\la|-1} p^{\sum_{i=1}^{\nu_1} \mu_{i}' (\nu_i' - \mu_i')} = p^{\sum_i \binom{\la_i'}{2}}.$$
Indeed, as $|\nu| -|\mu|=1$, we have $\nu_i' = \mu_i'$ for all $i$ except at one index $i_0$ where $\nu_{i_0}' - \mu_{i_0}'=1$. In this case $\sum_{i=1}^{\nu_1} \mu_{i}' (\nu_i' - \mu_i') = \mu_{i_0}'$. As such, $\prod_{j=1}^{|\la|-1} p^{\sum_{i=1}^{\nu_1} \mu_{i}' (\nu_i' - \mu_i')} = p^{[(\la_1'-1)+\dots+1] + [(\la_2'-1)+\dots+1] + \dots}  = p^{\sum_i \binom{\la_i'}{2}}$.

Finally, if $\la_1 \le d$, the total of such nested partition sequences ${\la^{(1)}}<\dots < {\la^{(i)}} < {\la^{(i+1)}}< \dots < {\la^{(|\la|)}}=\la$ is trivially bounded from above by $d^{|\la|}$.
\end{proof}

Before we prove \Cref{thm:matrix_product_intro} we mention two loose ends with the definition we gave in \Cref{def:cL_intro} of our limiting random variables. The first is that in \Cref{def:cL_intro} we stated that the random variable $\cL_{d,p^{-1},\chi}$ is $\Sig_d$-valued, while \Cref{thm:group_convergence_intro} (which we will shortly invoke) only has the power to show that a unique $\bSig_d$-valued random variable exists with a given set of moments. We will have to prove this simultaneously with \Cref{thm:matrix_product_intro}, so that all objects in the theorem are well-defined. Along the way we also check that our definition agrees with the quite different one given previously in \cite{van2023local}.

\begin{prop}\label{thm:cL_the_same}
\Cref{def:cL_intro} defines a unique $\Sig_d$-valued random variable. Furthermore, this random variable is the same as $\cL_{d,t,\chi}$ as defined in \cite[Theorem 6.1]{van2023local} when $t=1/p$.
\end{prop}


\begin{proof}[Proof of \Cref{thm:matrix_product_intro} and \Cref{thm:cL_the_same}]
We first show that there exists a unique $\bSig_d$-valued random variable $\cL_{d,p^{-1},\chi}$ solving the moment problem in \Cref{def:cL_intro} and that \Cref{thm:matrix_product_intro} holds for this random variable. Then we show the second part of \Cref{thm:cL_the_same}, which implies the first part, and thus shows that \Cref{def:cL_intro} makes sense and we have actually proven \Cref{thm:matrix_product_intro} holds as stated.

Letting $G_n$ be as in \Cref{thm:matrix_product_intro} and $c_n := \log_p k(n)$, we claim that
\begin{equation}\label{eq:limit_moments_no_zeta}
\lim_{n \to \infty} \frac{\E[\#\Hom(G_n,G_{\la'})]}{p^{|\la|c_n}} = \frac{n_{\max}(G_{\la'})}{|\la|!}.
\end{equation}
By factoring homomorphisms into a surjection and an injection, for any $\la,\nu \in \Sig_d^{\geq 0}$ we have
\begin{equation}\label{eq:hom_factor}
\#\Hom(G_{\nu'},G_{\la'}) = \sum_{\substack{\mu \in \Sig_d^{\geq 0}: \\ \mu_i \leq \la_i \text{ for all }i}} \frac{\#\Sur(G_{\nu'},G_{\mu'})\#\Inj(G_{\mu'},G_{\la'})}{\#\Aut(G_{\mu'})}.
\end{equation}
To show \eqref{eq:limit_moments_no_zeta} it suffices to show 
\begin{equation}\label{eq:limit_surj_no_zeta}
\lim_{n \to \infty} \frac{\E[\#\Sur(G_n,G_{\la'})]}{k(n)^{|\la|}} = \frac{n_{\max}(G_{\la'})}{|\la|!}
\end{equation}
(the denominators are the same in both, we have just rewritten it), since if we expand $\#\Hom(\cdot,\cdot)$ in \eqref{eq:limit_moments_no_zeta} using \eqref{eq:hom_factor}, we obtain a finite number of terms with coefficients depending only on the fixed signature $\la$, and if \eqref{eq:limit_surj_no_zeta} holds for all $\la$ then only the leading term in this sum survives.

As $k(n) \le e^{(\log n)^{1-\eps}}$, Corollary \ref{cor:sur:uniform:asym} is applicable and implies that 
\begin{equation}\label{eq:apply_sur_comp}
\Big |\E (\# \Sur(G_n),G_{\la'})) - n_{k(n)}(G_{\la'})\Big| \le e^{-n^{\al/(64p^{2|\la'|})}} n_{k(n)}(G_{\la'}).
\end{equation}
By \eqref{eq:limit_sbgp_seq} in \Cref{lemma:n_k:largek}, $n_{k(n)}(G_{\la'})/k(n)^{|\la|} = O(1)$ as $n \to \infty$, hence 
\begin{equation}\label{eq:compare_sur}
\frac{\E[\#\Sur(G_n,G_{\la'})]}{k(n)^{|\la|}} = (1+o(1))\frac{n_{k(n)}(G_{\la'})}{k(n)^{|\la|}}
\end{equation}
as $n \to \infty$ due to the exponential factor in \eqref{eq:apply_sur_comp}. Now, \eqref{eq:compare_sur} and \eqref{eq:limit_sbgp_seq} together imply \eqref{eq:limit_surj_no_zeta}, which shows \eqref{eq:limit_moments_no_zeta} as discussed above.

We now claim that the constants
\begin{equation}
C_\la := p^{-\zeta|\la|}\frac{n_{\max}(G_{\la'})}{|\la|!}
\end{equation} 
are nicely-behaved in the sense of \Cref{def:nicely_behaved}, where $\zeta$ is as in \Cref{thm:matrix_product_intro}. The upper bound in \Cref{lemma:n_k:largek} gives a bound 
\begin{equation}
p^{-\zeta|\la|}\frac{n_{\max}(G_{\la'})}{|\la|!} \leq F \frac{p^{\frac{\la_1^2}{2}+ \text{const} \cdot \la_1}}{|\la|!},
\end{equation}
and the $1/|\la|!$ provides the needed superlinear function in the exponent. Hence \Cref{thm:group_convergence_intro} applies after passing to the subsequence $(n_j)_{j \geq 1}$ of \Cref{thm:matrix_product_intro} so that $-c_{n_j}$ converges modulo $\Z$ as $j \to \infty$, with the $j$ in \Cref{thm:matrix_product_intro} playing the role of $n$ in \Cref{thm:group_convergence_intro}. 

Hence \Cref{thm:group_convergence_intro} applies, so there exists a unique $\bSig_d$-valued random variable with moments $C_\la$, and $(\rank(p^{i-1}G_{n_j})-\near{\log _{p}(k(n_{j}))+\zeta})_{1 \leq i \leq d}$ converges to it. When $\chi = p^{-\zeta}/(p-1)$, the moments $C_\la$ are exactly the ones in \Cref{def:cL_intro}. We refer to this limiting random variable as $\cL_{d,p^{-1},p^{-\zeta}/(p-1)}$, even though we have not yet shown that it is supported on $\Sig_d$ as claimed in \Cref{def:cL_intro}. We will do this now. The random variable in \cite[Theorem 6.1]{van2023local} is $\Sig_d$-valued by definition, so it suffices to show they are the same


To avoid confusion, we refer to the random variable called $\cL_{d,t,\chi}$ in \cite[Theorem 6.1]{van2023local} as $\tcL_{d,t,\chi}$. Taking the matrices in the setup of \Cref{thm:matrix_product_intro} to be distributed by the additive Haar measure on $\Z_p$, we have by \cite[Theorem 10.1]{van2023local}\footnote{Our $k(n)$ corresponds to $s_N$ in that result. Also, \cite[Theorem 10.1]{van2023local} is stated in terms of the singular numbers of the matrix, see \cite[(1.12)]{van2023local} for the equivalence with our notation.} that 
\begin{equation}\label{eq:cite_10.1_otherpaper}
(\rank(p^{i-1}G_{n_j})-\near{\log _{p}(k(n_{j}))+\zeta})_{1 \leq i \leq d} \to \tcL_{d,p^{-1},p^{-\zeta}/(p-1)}
\end{equation}
in distribution as $j \to \infty$. Though we wrote the initial argument of this proof for random matrices over $\Z$, the same one applies over $\Z_p$ because \Cref{cor:sur:uniform:asym} holds in either setting. Combining this fact with \eqref{eq:cite_10.1_otherpaper} yields that $\cL_{d,p^{-1},\chi} = \tcL_{d,p^{-1},\chi}$ for all $\chi \in \R_{>0}$, completing the proof. 
\end{proof}

\begin{rmk}
Although our proof works for some sequences $k(n) \to \infty$, the growth is rather slow, while in the additive Haar case of \cite[Theorem 10.1]{van2023local}, the result applies to any $k(n)$ satisfying the condition $k(n) \ll p^n$. This is mainly due to the limitation of the method in Section \ref{section:prod:moments}, and it is an interesting problem to extend the range to $p^{cn}$ for some $c<1$. 
\end{rmk}

\begin{rmk}
As mentioned in the Introduction, one may also make sense of \Cref{def:cL_intro} for a general real parameter $p > 1$ by extrapolating $n_{\max}$ as a function of $p$, and one may check this agrees with the definition in \cite{van2023local} for all such real $p$, not necessarily prime. However, for now we resist this diversion.
\end{rmk}

\subsection{Acknowledgements} We thank Mehtaab Sawhney for helpful remarks on his work and related ones, and thank the anonymous referees for many helpful corrections and comments. HN is supported by National Science Foundation grant DMS-1752345 and by a Simons Fellowship. RVP was supported by the European Research Council (ERC), Grant Agreement No. 101002013.


\end{document}